\documentclass[12 points]{article}
\usepackage{amssymb}
\usepackage{eucal}
\usepackage{amsmath}
\usepackage{amsthm}
\usepackage{graphicx}
\usepackage{float}
\usepackage{framed}
\usepackage{blkarray}
\usepackage{enumitem}

\usepackage[title]{appendix}

\usepackage{relsize}

\newcommand{\R}{{\mathbb  R}}  \numberwithin{equation}{section} \newtheorem{thm}{\bf
Theorem}[section]
\newtheorem{lem}[thm]{\bf Lemma} \newtheorem{prop}[thm]{\bf Proposition} \newtheorem{cor}[thm]{\bf Corollary}
  \theoremstyle{remark}
\newtheorem{rem}{\bf
Remark}[section]  \newtheorem{exmp}{\bf Example}[section]

\begin{document}

\title{\huge \bf Lyapunov matrix equation $A^HX+XA+C=O_r$ with $A$ a Jordan matrix} 
 \author{Dan Com\u{a}nescu \\ {\small Department of Mathematics, West University of Timi\c soara}\\ {\small Bd. V.
P\^ arvan,
No 4, 300223 Timi\c soara, Rom\^ania}\\ {\small dan.comanescu@e-uvt.ro}\\ }
\date{ }

 \maketitle
 
 \begin{abstract}
 A Lyapunov matrix equation can be converted, by using the Jordan decomposition theorem for matrices,  into an equivalent Lyapunov matrix equation where the matrix is a Jordan matrix. 
The Lyapunov matrix equation with Jordan matrix can be reduced to a system of Sylvester-Lyapunov type matrix equations. We completely solve the Sylvester-Lyapunov type matrix equations corresponding to the Jordan block matrices of the initial matrix.
 \end{abstract}



\section{Introduction}

A Lyapunov matrix equation (or the continuous Lyapunov matrix equation) has the form
\begin{equation}\label{Lyapunov-matrix-equation}
A^HX+XA+C=O_{r},
\end{equation} 
with $A\in \mathcal{M}_{r}(\mathbb{C})$, $C\in \mathcal{M}_{r}(\mathbb{C})$ and the unknown $X\in \mathcal{M}_{r}(\mathbb{C})$. 

{ The homogeneous Lyapunov matrix equation} is:
\begin{equation}\label{Lyapunov-matrix-equation-zero}
A^HX+XA=O_{r}.
\end{equation} 

The Lyapunov matrix equation appeared, see \cite{khalil} and \cite{halanay-rasvan}, in the study of the dynamics of a linear differential equation, 
$\dot{\bf x}=A{\bf x}$,
where $A\in \mathcal{M}_{r}(\mathbb{C})$.
Dynamical aspects, like stability, are studied using a square function $V({\bf x})=\frac{1}{2}{\bf x}^HX{\bf x}$, with $X\in \mathcal{M}_{r}(\mathbb{C})$ a Hermitian matrix. The function $V$ is a conserved quantity of the linear system if and only if $X$ is a Hermitian solution of the homogeneous Lyapunov matrix equation \eqref{Lyapunov-matrix-equation-zero}.
The global asymptotic stability of the equilibrium point ${\bf 0}$ is assured by the existence of a Hermitian positive definite solution of the Lyapunov matrix equation \eqref{Lyapunov-matrix-equation}, with $C$ a Hermitian positive definite matrix.

Lyapunov matrix equations appear also in the study of optimization of a cost function defined on a matrix manifold, see \cite{birtea-casu-comanescu}.

The Lyapunov matrix equation is a particular case of the Sylvester matrix equation 
\begin{equation}\label{Sylvester-matrix-equation-202}
AX+XB+C=O_{r},
\end{equation} 
with $A\in \mathcal{M}_{r}(\mathbb{C})$, $B\in \mathcal{M}_{s}(\mathbb{C})$, $C\in \mathcal{M}_{r\times s}(\mathbb{C})$ and the unknown $X\in \mathcal{M}_{r\times s}(\mathbb{C})$. 

There are numerous studies of this equation in the literature, see \cite{gantmacher,lancaster-tismenetsky,roth}, and we find some representations of the solutions for the Sylvester matrix equation, see \cite{jameson}, but  they are difficult to use when $r$ and $s$ are large numbers. We also find studies of generalized Sylvester matrix equation, see \cite{duan}, or studies of other matrix equations which are derived from the Lyapunov matrix equation, see \cite{braden} and \cite{song-feng-zhao}. Due to the importance of the Lyapunov matrix equations, researchers have made a great effort in studying its numerical solutions, see  \cite{hammarling,simoncini}.

The set of solutions of a homogeneous Lyapunov matrix equation defines some types of matrices $A$:

$\bullet$ if the identity matrix is a solution of \eqref{Lyapunov-matrix-equation-zero}, then $A$ is a skew-Hermitian matrix;

$\bullet$ if the matrix $\begin{pmatrix}
O_r & I_r \\
-I_r & O_r
\end{pmatrix}$ is a solution of \eqref{Lyapunov-matrix-equation-zero}, then $A$ is a Hamiltonian matrix;

$\bullet$ if \eqref{Lyapunov-matrix-equation-zero} has an invertible solution, then $A^H$ and $-A$ are similar matrices.\medskip

In the papers \cite{gantmacher} and \cite{lancaster-tismenetsky} the Sylvester matrix equation is converted into another equivalent Sylvester matrix equation, by using the Jordan decomposition theorem for matrices, in which the matrices $A$ and $B$ are replaced by Jordan matrices.
In this paper we convert a Lyapunov matrix equation, by using also the Jordan decomposition theorem for matrices,  into an equivalent Lyapunov matrix equation with the matrix $A$ being replaced with a Jordan matrix. We note that our conversion (used for a Lyapunov matrix equation) is different from the conversion proposed in \cite{gantmacher} and \cite{lancaster-tismenetsky}.

In Section \ref{Lyapunov matrix equation. Generalities} we present some general results concerning the set of solutions of Lyapunov matrix equation, the set of Hermitian solutions and the set of real symmetric solutions. Also, we present the  equivalent transformation of   \eqref{Lyapunov-matrix-equation} by using a similar transformation of the matrix $A$, which represents, via Jordan canonical form, the main reason to study the Lyapunov matrix equation with Jordan matrix.

 The Lyapunov matrix equation with Jordan matrix is reduced to a system with  Sylvester-Lyapunov matrix equations with Jordan block matrices. In Section \ref{section-Jordan block matrices} we completely solve a more general equation as Lyapunov matrix equation with Jordan block matrix, called Sylvester-Lyapunov matrix equation with Jordan block matrices. For the homogeneous case we present the set of solutions, the set of Hermitian solutions, and the set of real symmetric solutions. In the non-homogeneous case with a Jordan block matrix which is not nilpotent we have a unique solution which is written by using the Pascal matrix and the extended generalized Pascal matrix introduced in \cite{zhang-liu}. For the case of square matrices  we study the solution when the matrix $C$ is positive semidefinite or positive definite.
 When both Jordan block matrices are nilpotent we present the compatibility conditions and, if they are satisfied, the set of solutions, the set of Hermitian solutions, and the set of real symmetric solutions.

In Section \ref{Lyapunov matrix equation with Jordan matrix} we present and study the equivalent system formed with Sylvester-Lyapunov matrix equations which is attached to a Lyapunov matrix equation with a Jordan matrix. For the homogeneous case we present the dimension of the complex vectorial space of solutions and the dimension of the real vectorial spaces of Hermitian solutions and real symmetric solutions. As an example we present these sets for the case of a diagonal matrix. In the homogeneous case we study the conditions for which we have a Hermitian positive definite solution and the conditions for which we have an invertible solution. We present two examples, in the homogeneous case and in the non-homogeneous case, to better understand the used method.

Section \ref{Lyapunov matrix equation. General case.} is dedicated to the general Lyapunov matrix equation. In the homogeneous case, if we know the set of the eigenvalues of the matrix $A$, by using the Jordan decomposition theorem for matrices and our results from the previous sections, we compute the dimension of the complex vectorial space of solutions, the dimension of the real vectorial space of the Hermitian solutions, and the dimension of the real vectorial space of the real symmetric solutions.  In this case we study the existence of the invertible solutions and the positive definite solutions. To compute the solutions in the general case we need to know the set of eigenvalues of $A$ and the matrix $P_A$ which appears in the Jordan decomposition theorem for matrices. We illustrate with  an example of solving the Lyapunov matrix equation in the general case  by the method presented in this paper.

In the Appendix we recall some notions and results of the matrix theory which appear in our paper.

\section{Lyapunov matrix equation. Generalities}\label{Lyapunov matrix equation. Generalities}

We denote by $\mathcal{L}_{A,C}$ the set of the solutions of \eqref{Lyapunov-matrix-equation}.
It is easy to observe the following result.

\begin{prop}\label{affine-property-general-case}
(i) $O_r\in \mathcal{L}_{A,O_r}$ and $\mathcal{L}_{A,O_r}$ is a complex vectorial space.

\noindent (ii)  If $Y_C$ is a particular solution of \eqref{Lyapunov-matrix-equation}, then we have $\mathcal{L}_{A,C}=Y_C+\mathcal{L}_{A,O_r}$.
\end{prop}

In some important situations the matrix $C$ is supposed to be a Hermitian matrix or real symmetric matrix. 
We note by $\mathcal{L}^{H}_{A,C}$ (respectively $\mathcal{L}^{sym}_{A,C}$) the set of Hermitian (respectively real symmetric) solutions of the equation \eqref{Lyapunov-matrix-equation}. 

\begin{prop}\label{Hermitian-solutions-abc}
 (i) If $\mathcal{L}^{H}_{A,C}\neq \emptyset$, then $C$ is a Hermitian matrix.

\noindent (ii) $O_{r}\in \mathcal{L}^{H}_{A,O_{r}}$ and 
$\mathcal{L}^{H}_{A,O_{r}}\subseteq \mathcal{L}_{A,O_{r}}$ is a real vectorial space.

\noindent (iii) Suppose that  $C$ is a Hermitian matrix. 

(iii.1) If $X\in \mathcal{L}_{A,C}$, then $X^H\in \mathcal{L}_{A,C}$.

(iii.2) If $X\in \mathcal{L}_{A,C}$, then $X^h=\frac{1}{2}\left(X+X^H\right)\in \mathcal{L}^{H}_{A,C}$.

(iii.3) If $Y_C$ is a Hermitian solution of \eqref{Lyapunov-matrix-equation}, then $\mathcal{L}^{H}_{A,C}=Y_C+\mathcal{L}^{H}_{A,O_{r}}$.
\end{prop}

\begin{prop}\label{symmetric-solutions-abc}
Suppose that $A,C$ are real matrices.  

\noindent (i) If $X\in \mathcal{L}_{A,C}$, then $\emph{Re}(X)\in \mathcal{L}_{A,C}$.

\noindent (ii) If $\mathcal{L}^{sym}_{A,C}\neq \emptyset$, then $C$ is a real symmetric matrix.

\noindent (iii) $O_r\in \mathcal{L}^{sym}_{A,O_r}$ and  $\mathcal{L}^{sym}_{A,O_r}\subseteq \mathcal{L}^{H}_{A,O_r}$ is a real vectorial space.

\noindent (iv) Suppose that $C$ is a real symmetric matrix. 

(iv.1) If $X\in \mathcal{L}_{A,C}$ is a real matrix, then $X^T\in \mathcal{L}_{A,C}$.

(iv.2) If $X\in \mathcal{L}_{A,C}$ is a real matrix, then $X^s=\frac{1}{2}\left(X+X^T\right)\in \mathcal{L}^{sym}_{A,C}$.

(iv.3)  If $Z_C$ is a real symmetric solution of \eqref{Lyapunov-matrix-equation}, then $\mathcal{L}^{sym}_{A,C}=Z_C+\mathcal{L}^{sym}_{A,O_r}.$
\end{prop}

The existence of an invertible solution of the homogeneous equation \eqref{Lyapunov-matrix-equation-zero} is equivalent with the fact that matrices $A^H$ and $-A$ are similar.

In some cases we can simplify the Lyapunov matrix equation by using a matrix  similar with the matrix $A$.

\begin{prop}\label{equivalence-Jordan-102}
Let $B\in \mathcal{M}_r(\mathbb{C})$ be a matrix  similar  with $A$ via the invertible matrix $P\in \mathcal{M}_r(\mathbb{C})$ (i.e. $B=P^{-1}AP$) and $D=P^HCP$. We construct the Lyapunov matrix equation
\begin{equation}\label{Lyapunov-matrix-equation-Z}
B^HZ+ZB+D=O_r.
\end{equation}
The matrix $X$ is a solution of \eqref{Lyapunov-matrix-equation} if and only if $Z=P^HXP$ is a solution of \eqref{Lyapunov-matrix-equation-Z}.
\end{prop}

\begin{proof}
We have 
\begin{align*}
& A^HX+XA+C=O_r \Leftrightarrow P^{-H}B^HP^HX+XPBP^{-1}+C=O_r \\
 \Leftrightarrow & B^HP^HX+P^HXPBP^{-1}+P^HC=O_r \\
  \Leftrightarrow & B^HP^HXP+P^HXPB+P^HCP=O_r 
\Leftrightarrow  B^HZ+ZB+D=O_r.\qed
\end{align*}
\end{proof}

We have the following properties:

\noindent $\bullet$ The equation  \eqref{Lyapunov-matrix-equation} is homogeneous if and only if the equation  \eqref{Lyapunov-matrix-equation-Z} is homogeneous.

\noindent $\bullet$  $C$ is a Hermitian matrix if and only if $D=P^HCP$ is a Hermitian matrix.

\noindent $\bullet$ If $P$ is a real matrix, then $C$ is a real symmetric matrix if and only if $D$ is a real symmetric matrix.

\section{Sylvester-Lyapunov matrix equations with Jordan block matrices}\label{section-Jordan block matrices}

In this section we completely solve { the Sylvester-Lyapunov matrix equation with Jordan block matrices}, a matrix equation of the form:
\begin{equation}\label{Sylvester-Jordan-block}
J_r(\lambda)^HX+XJ_s(\mu)+C=O_{r,s},
\end{equation}
where $J_r(\lambda)\in \mathcal{M}_r(\mathbb{C}),\, J_s(\mu)\in \mathcal{M}_s(\mathbb{C})$ are Jordan block matrices, $C\in \mathcal{M}_{r,s}(\mathbb{C})$ and the unknown matrix $X\in \mathcal{M}_{r,s}(\mathbb{C})$.
We denote by $\mathcal{L}(\lambda,\mu,C)$ the set of solutions of \eqref{Sylvester-Jordan-block}. $\mathcal{L}^H(\lambda,\mu,C)$ is the set of Hermitian solutions and $\mathcal{L}^{sym}(\lambda,\mu,C)$ is the set of real symmetric solutions.

\begin{rem}
In the paper \cite{gantmacher} it is studied the following matrix equation:
\begin{equation}
J_r(\lambda )X+XJ_s(\mu )+C=O_{r,s},
\end{equation}
which is different from \eqref{Sylvester-Jordan-block}.
\end{rem}

\begin{prop}\label{Prop-Sylvester-multiplu}
 (i) If $\alpha\in \mathbb{C}^*$ and $C\in \mathcal{M}_{r,s}(\mathbb{C})$, then $\mathcal{L}(\lambda,\mu,\alpha C)=\alpha \mathcal{L}(\lambda,\mu,C).$

\noindent (ii) If $C,C_1,\dots, C_q\in \mathcal{M}_{r,s}(\mathbb{C})$ such that $C=\sum\limits_{i=1}^qC_i$, and $\mathcal{L}(\lambda,\mu,C_1)\neq \emptyset$, ... , $\mathcal{L}(\lambda,\mu,C_q)\neq \emptyset$, then 
$\mathcal{L}(\lambda,\mu,C)=\sum\limits_{i=1}^q\mathcal{L}(\lambda,\mu,C_i).$

\noindent (iii) If $\nu=\lambda+\overline{\mu}$, then the equation \eqref{Sylvester-Jordan-block} is equivalent with the matrix equation
\begin{equation}\label{Sylvester-Jordan-block-particular}
J_r(\nu)^HX+XJ_s(0)+C=O_{r,s},
\end{equation}
 i.e. $\mathcal{L}(\lambda,\mu,C)=\mathcal{L}(\nu,0,C)$.
\end{prop}

\begin{proof}
To prove $(iii)$ we use the relation \eqref{Jordan-block-zero-relation}. \qed
\end{proof}
\medskip

The set of matrices 
$\{{\bf e}_{\alpha}\otimes {\bf f}^T_{\beta}\in \mathcal{M}_{r,s}(\mathbb{C})\,|\,\alpha=\overline{1,r},\,\beta=\overline{1,s}\}$
is a basis of the complex vectorial space $\mathcal{M}_{r,s}(\mathbb{C})$, where ${\bf e}_1,\dots,{\bf e}_r$ is the canonical basis of $\mathbb{C}^r$ and ${\bf f}_1,\dots,{\bf f}_s$ is the canonical basis of $\mathbb{C}^s$.
The above proposition suggests to consider the homogeneous case and the cases when $C$ is an element of the above basis of $\mathcal{M}_{r,s}(\mathbb{C})$.\medskip

In what follows we study the equation \eqref{Sylvester-Jordan-block-particular}. This equation written  in components of $X$ is given by
\begin{equation}\label{Sylvester-recurenta}
\nu x_{ij}+x_ {i-1,j}+x_{i,j-1}+c_{ij}=0,\,\,\,\forall i\in \{1,\dots, r\},\,\forall j\in \{1,\dots, s\},
\end{equation}
where $x_{0j}=x_{i0}=0$.

\subsection{The homogeneous case}

In this case we have $C=O_{r,s}$ and the equation \eqref{Sylvester-Jordan-block-particular} becomes
\begin{equation}\label{Sylvester-Jordan-block-particular-zero}
J_r(\nu)^HX+XJ_s(0)=O_{r,s}.
\end{equation}
\begin{thm}\label{homogeneous-Jordan-block-F101} 
(i) If $\nu\neq 0$, then $\mathcal{L}(\nu,0,O_{r,s})=\{O_{r,s}\}$. 

\noindent (ii) $\mathcal{L}(0,0,O_{r,s})$ is a complex vectorial space with $\dim_{\mathbb{C}}\mathcal{L}(0,0,O_{r,s})=\min\{r,s\}$.
A solution $X$ of \eqref{Sylvester-Jordan-block-particular-zero} has the components:
\begin{align*}
x_{ij}= & 
\begin{cases}
0 & i+j\leq s \\
(-1)^{s-j}p_{i+j-s} & i+j>s
\end{cases}\,{if}\,r\leq s; \\
\,x_{ij}= &
\begin{cases}
0 & i+j\leq r \\
(-1)^{r-i}p_{i+j-r} & i+j>r
\end{cases}\,{if}\,r>s,
\end{align*}
where $p_1,\dots,p_{\min\{r,s\}}\in \mathbb{C}$. We can write, see \eqref{Yt-formula} and \eqref{Y-identity-98},
\begin{equation}
X=\sum\limits_{t=1}^{\min\{r,s\}}p_t \mathcal{Y}_t^{[r,s]}=
\begin{cases}
\left(\sum\limits_{t=1}^{r}p_t(J_r^T(0))^{t-1}\right)\mathcal{Y}^{[r,s]} & {if}\,\,\,r\leq s \\
\mathcal{Y}^{[r,s]}\left(\sum\limits_{t=1}^{s}p_t(J_s(0))^{t-1}\right) & {if}\,\,\,r> s
\end{cases}
,
\end{equation}
where the set of matrices $\left\{\mathcal{Y}_1^{[r,s]},\dots, \mathcal{Y}_{\min\{r,s\}}^{[r,s]}\right\}$ is a basis of $\mathcal{L}(0,0,O_{r,s})$.
\end{thm}

\begin{proof}
$(i)$ We use the relations \eqref{Sylvester-recurenta} and mathematical induction.

$(ii)$ First we suppose that $r\leq s$.
We choose $i=1$ in \eqref{Sylvester-recurenta} and obtain that $x_{11}=\dots=x_{1,s-1}=0$. By mathematical induction $x_{i,j}=0$ for $i+j\leq s$. For an index $i$ we set $x_{is}=p_i$. By induction we obtain that $x_{i+j,s-j}=(-1)^j p_{i}$ for $j\in \{1,\dots,r-i\}$. All these matrices are solutions for \eqref{Sylvester-Jordan-block-particular-zero}. We observe that the set of solutions is a vectorial space with the dimension $r=\min\{r,s\}$.

For the case $r>s$  we can consider the equivalent problem 
\begin{equation*}
J_s(0)^HX^H+X^HJ_r(0)=O_{s,r}
\end{equation*}
and using the above results we find the announced results. \qed
\end{proof}

\begin{cor}\label{solutie-patrat-1034}
For $\nu=0$, $r=s$
a solution of the equation \eqref{Sylvester-Jordan-block-particular-zero} has the form 
\begin{equation}\label{homogeneous-solutions-patrat}
X=\begin{pmatrix}
0 & 0 & \dots & 0 & 0 & p_{1} \\
0 & 0 & \dots & 0 & -p_{1} & p_{2} \\
0 & 0 & \dots & p_{1} & -p_{2} & p_{3} \\
\vdots & \vdots & \ddots & \vdots & \vdots & \vdots \\
(-1)^{r-1}p_{1} & (-1)^{r-2}p_{2} & \dots & p_{r-2} & -p_{r-1} & p_{r} 
\end{pmatrix},
\end{equation}
where $p_1,\dots,p_r\in \mathbb{C}$.
The matrix $X$
is invertible if and only if $p_1\neq 0$.
\end{cor}

\begin{proof}
We observe that $\det(X)=p_1^r$ and we obtain the announced result. \qed
\end{proof}

\begin{cor}\label{Lyapunov-general-invertible-101-hermitian}
The matrix $X\in \mathcal{L}^H(0,0,O_{r})$
has the form
$$X=\sum\limits_{t=1}^{r}u_t \mathfrak{i}^t\mathcal{Y}_t^{[r,r]}\,\,{if}\,\,r\in 2\mathbb{N}\,\,{and}\,\,X=\sum\limits_{t=1}^{r}u_t \mathfrak{i}^{t-1}\mathcal{Y}_t^{[r,r]}\,\,{if}\,\,r\in 2\mathbb{N}+1,$$
where $u_1,\dots,u_r$ are real parameters.
$\mathcal{L}^H(0,0,O_r)$  is a real vectorial space with $\dim_{\mathbb{R}}\mathcal{L}^H(0,0,O_r)=r$. The matrix $X$ is invertible if and only if $u_1\neq 0$.

The set $\mathcal{L}^H(0,0,O_r)$ contains a positive definite matrix if and only if $r=1$.
\end{cor}

\begin{proof} We present the proof for the last affirmation.
For $r=1$ it easy to see that we have a positive definite matrix in $\mathcal{L}^H(0,0,O_r)$.
If $r>1$, then an element $X$ of $\mathcal{L}^H(0,0,O_r)$ has the component $x_{11}=0$. We have ${\bf e}^H_1X{\bf e}_1=x_{11}=0$, which implies that $X$ is not positive definite. \qed
\end{proof}

\begin{cor}\label{Lyapunov-general-invertible-1012-symmetric} 
The set $\mathcal{L}^{sym}(0,0,O_r)$ of the real symmetric solutions of the equation \eqref{Sylvester-Jordan-block-particular-zero} is a real vectorial space with  $\dim_{\mathbb{R}}\mathcal{L}^H(0,0,O_r)=\left[\frac{r+1}{2}\right]$\footnote{We use the floor function $[\,\cdot\,]$.}. A real symmetric solution has the form
$$X=\sum\limits_{t=1}^{\frac{r}{2}}u_{2t} \mathcal{Y}_{2t}^{[r,r]}\,\,{if}\,\,r\in 2\mathbb{N}\,\,{and}\,\,X=\sum\limits_{t=1}^{\frac{r}{2}}u_{2t} \mathcal{Y}_{2t}^{[r,r]}\,\,{if}\,\,r\in 2\mathbb{N}+1,$$
where $u_1,u_2,\dots,u_r\in \R$. $X$ is invertible if and only if $r\in 2\mathbb{N}+1$ and $u_1\neq 0$.
\end{cor}

\begin{cor}
Let be $\lambda,\mu\in \mathbb{C}$ and $r\in \mathbb{N}^*$. 

(i) The matrices $J^H_r(\lambda)$ and $J_r(\mu)$ are similar if and only if $\lambda+\overline{\mu}=0$.

(ii) If $\lambda+\overline{\mu}=0$, then $J^H_r(\lambda)$ and $J_r(\mu)$ are similar via a Hermitian matrix.

(iii) If $\lambda+\overline{\mu}=0$, then $J^H_r(\lambda)$ and $J_r(\mu)$ are similar via a real symmetric matrix if and only if $r\in 2\mathbb{N}+1$.
\end{cor}

\begin{cor}
$J_{2r}(\lambda)$ is a Hamiltonian matrix if and only if $r=1$ and $\lambda$ is a purely imaginary number.
\end{cor}

\begin{proof} Suppose that $J_{2r}(\lambda)$ is a Hamiltonian matrix. We consider the homogeneous matrix equation
$J^H_{2r}(\lambda)X+XJ_{2r}(\lambda)=O_{2r}.$
This equation is equivalent with the following homogeneous Lyapunov matrix equation
$J^H_{2r}(\lambda+\overline{\lambda})X+XJ_{2r}(0)=O_{2r}.$
A solution is given by the matrix $\begin{pmatrix}
O_r & I_r \\
-I_r & O_r
\end{pmatrix}$. We obtain that $\lambda+\overline{\lambda}=0$.
By using the form of the solutions we deduce that $r=1$.

We observe that $J_2(a \mathfrak{i})$, $a\in \R$, is a Hamiltonian matrix. \qed
\end{proof}

\subsection{The non-homogeneous case with $\nu\neq 0$}

\begin{thm}\label{nonhomogeneous-11-Jordan-block} If $\nu\neq 0$ and $C\in \mathcal{M}_{r\times s}(\mathbb{C})$, then  \eqref{Sylvester-Jordan-block-particular} has a unique solution. 

\noindent (a) If $C={\bf e}_i\otimes {\bf f}_j^T$, with $(i,j)\in \{1,\dots,r\}\times\{1,\dots,s\}$, then the solution is $\mathcal{X}^{[r,s]}_{ij}\left[-\frac{1}{\nu}\right]$ (see Section \ref{appendix}, \eqref{X-identity-998}).

\noindent (b) For a general $C$ the solution is $X=\sum\limits_{i=1}^r\sum\limits_{j=1}^s c_{ij}\mathcal{X}^{[r,s]}_{ij}\left[-\frac{1}{\nu}\right]$. We can write
\begin{equation}\label{nu-diferit-zero-reprezentare-11}
X=\sum\limits_{i=1}^r\sum\limits_{j=1}^s c_{ij}(J_r^T(0))^{i-1}\mathcal{X}^{[r,s]}\left[-\frac{1}{\nu}\right](J_s(0))^{j-1}.
\end{equation}

\end{thm}

\begin{proof} The unicity is proved by mathematical induction.

$(a)$ First, we consider the  matrix equation  \eqref{Sylvester-Jordan-block-particular} with $\nu\neq 0$ and $\widehat{C}=-\nu{\bf e}_1\otimes {\bf f}_1^T$. The solution $\widehat{X}$ has the component $\widehat{x}_{11}=1$ and the following recursion is verified:
\begin{equation}
\widehat{x}_{ij}=-\frac{1}{\nu}\left(\widehat{x}_ {i-1,j}+\widehat{x}_{i,j-1}\right),\,\,\,\forall i\in \{1,\dots, r\},\,\forall j\in \{1,\dots, s\},\,(i,j)\neq (1,1).
\end{equation}
where $\widehat{x}_{0j}=\widehat{x}_{i0}=0$.

By using the properties of the matrix $\Psi^{[r,s]}[-\frac{1}{\nu}]$, see Section \ref{appendix}, we obtain that $\widehat{X}=\Psi^{[r,s]}[-\frac{1}{\nu}]$.
The Proposition \ref{Prop-Sylvester-multiplu} $(i)$ implies the announced result.

Second, we consider $i=1$ and $j>1$ and we write $X=(Y\,Z)$ with $Y\in \mathcal{M}_{r,j-1}(\mathbb{C})$ and $Z\in \mathcal{M}_{r,s-j+1}(\mathbb{C})$. We observe that we can write $J_s(0)=J_{j-1}(0)\oplus J_{s-j+1}(0)$.
The equation \eqref{Sylvester-Jordan-block-particular} has the form
$$J_r^H(\nu)(Y\,Z)+(Y\,Z)\left(J_{j-1}(0)\oplus J_{s-j+1}(0)\right)+\left(O_{r,j-1}\,{\bf e}_1\otimes \widetilde{\bf f}^T_1\right)=\left(O_{r,j-1}\,O_{r,s-j+1}\right),$$
where $\widetilde{\bf f}_1\in \R^{s-j+1}$. 
It is equivalent with the matrix system
$$\begin{cases}
J_r^H(\nu)Y+YJ_{j-1}(0)=O_{r,j-1} \\
J_r^H(\nu)Z+ZJ_{s-j+1}(0)+{\bf e}_1\otimes \widetilde{\bf f}^T_1=O_{r,s-j+1}.
\end{cases}$$
By using Theorem \ref{homogeneous-Jordan-block-F101} and first step of the induction, we obtain
$Y=O_{r,j-1},\,\,\,Z=-\frac{1}{\nu}\Psi^{[r,s-j+1]}[-\frac{1}{\nu}],$
which implies our result.

Analogously we obtain the announced results for the remaining cases.

$(b)$ We observe that $C=\sum_{i=1}^r\sum_{j=1}^s c_{ij}{\bf e}_i\otimes {\bf f}_j^T$ and we apply Proposition \ref{Prop-Sylvester-multiplu} and the above results.
For the formula \eqref{nu-diferit-zero-reprezentare-11} we use \eqref{X-identity-998}. \qed
\end{proof}
\medskip

In some important cases we have that $C$ is a Hermitian positive semidefinite matrix. There exists $L\in \mathcal{M}_{r\times s}(\mathbb{C})$ such $C=LL^H\in \mathcal{M}_r(\mathbb{C})$, see Theorem 7.2.7 from \cite{horn}. The components verify $c_{ij}=\sum\limits_{q=1}^sl_{iq}\overline{l_{jq}}$. The solution presented in the above theorem has the form
\begin{equation}\label{X-Cholesky-10}
X=\sum\limits_{q=1}^s L_q\mathcal{X}^{[r,r]}[-\frac{1}{\nu}]L_q^H,
\end{equation}
where 
$$L_q=\sum\limits_{i=1}^r l_{i,q}(J_r^T(0))^{i-1}=
\begin{pmatrix}
l_{1,q} & 0 & 0 & \dots & 0 \\
l_{2,q} & l_{1,q} & 0 & \dots & 0 \\
\vdots & \vdots & \vdots & \ddots & \vdots \\
l_{r,q} & l_{r-1,q} & l_{r-2,q} & \dots & l_{1,q}
\end{pmatrix}.
$$

\begin{thm}
Suppose that  $\nu\in \R^*$ and $C$ is a Hermitian positive semidefinite matrix. 

(i) If $\nu>0$ (respectively $\nu<0$), then the solution of \eqref{Sylvester-Jordan-block-particular} is a negative (respectively positive) semidefinite matrix.

(ii) If $\nu>0$ (respectively $\nu<0$) and $c_{11}\neq 0$, then the solution is a negative (respectively positive) definite matrix.

\end{thm}

\begin{proof} We can write $C=LL^H\in \mathcal{M}_r(\mathbb{C})$ with $L\in \mathcal{M}_{r}(\mathbb{C})$. The solution is a Hermitian matrix.

$(i)$ In this case $-\frac{1}{\nu}<0$ and the matrix $\mathcal{X}^{[r,r]}[-\frac{1}{\nu}]$ is negative definite, see Section \ref{appendix}. By using Observation 7.1.8 from \cite{horn} we have that for $q\in \{1,\dots,r\}$ the matrix $L_q\mathcal{X}^{[r,r]}[-\frac{1}{\nu}]L_q^H$ is negative semidefinite. Using now Observation 7.1.3 from \cite{horn} we deduce that the solution is negative semidefinite.

The case $\nu<0$ is treated analogously.

$(ii)$ We observe that $c_{11}=\sum_{i=1}^r|l_{1q}|^2$. If $c_{11}\neq 0$, then it exists $q^*\in \{1,\dots,r\}$ such that $l_{1q^*}\neq 0$. We have $\text{rank}(L_{q^*})=r$ and we use Observation 7.1.8 from \cite{horn} to deduce that the matrix $L_q\mathcal{X}^{[r,r]}[-\frac{1}{\nu}]L_q^H$ is positive definite. By using $(i)$ and Observation 7.1.3 from \cite{horn} we deduce that the solution is positive definite. 
The case $\nu<0$ is treated analogously. \qed

\end{proof}

\begin{cor}
Suppose that  $\nu\in \R^*$ and $C$ is a Hermitian positive definite matrix. 
If $\nu>0$ (respectively $\nu<0$), then the solution of \eqref{Sylvester-Jordan-block-particular} is a negative (respectively positive) definite matrix.

\end{cor}

\subsection{The non-homogeneous case with $\nu= 0$}

The components of a solution verify
\begin{equation}\label{Sylvester-recurenta-nu-zero}
x_ {i-1,j}+x_{i,j-1}+c_{ij}=0,\,\,\,\forall i\in \{1,\dots, r\},\,\forall j\in \{1,\dots, s\},
\end{equation}
where $x_{0j}=x_{i0}=0$.

\begin{thm}\label{Sylvester-Jordan-block-homogeneous-thm}
If $\nu=0$, then the equation \eqref{Sylvester-Jordan-block-particular} has solutions if and only if 
 \begin{equation}
 \sum\limits_{j=1}^{i-1}(-1)^jc_{i-j,j}=0,\,\,\,\forall i\in \{2,\dots,\min\{r,s\}+1\}.
 \end{equation}

\noindent If $\widehat{X}\in \mathcal{L}(0,0,C)$, then
$\mathcal{L}(0,0,C)=\widehat{X}+\mathcal{L}(0,0,O_{r,s}).$

\noindent If the above conditions are satisfied, then the components of a particular solution are:
\begin{align*}
& \text{for}\,\,r\leq s:\,\,\widehat{x}_{ij}=
\begin{cases}
\,\,\,\,\,\,\,\,\,0 & \text{if}\,\,j=s \\
\sum\limits_{k=1}^{\min\{i,s-j\}}(-1)^k c_{i-k+1,j+k} &  \text{if}\,\,j<s
\end{cases};\\
& \text{for}\,\,r> s:\,\,\widehat{x}_{ij}=
\begin{cases}
\,\,\,\,\,\,\,\,\,0 & \text{if}\,\,i=r \\
\sum\limits_{k=1}^{\min\{j,r-i\}}(-1)^k c_{i+k,j-k+1} &  \text{if}\,\,i<r
\end{cases}.
\end{align*}

\end{thm}

\begin{proof}
First we suppose that the matrix equation \eqref{Sylvester-Jordan-block-particular} has solutions. We write the equation \eqref{Sylvester-recurenta-nu-zero} for $i=j=1$ and we obtain $c_{11}=0$. For $i\geq 3$ we have
$$x_{i-2,1}+c_{i-1,1}=0,\,x_{i-3,2}+x_{i-2,1}+c_{i-2,2}=0,\dots,\,x_{1,i-2}+c_{1,i-1}=0$$
and we obtain the last equalities.

By calculus it is verified that the matrix $\widehat{X}$ from the statement of the Theorem is a particular solution of \eqref{Sylvester-recurenta-nu-zero}.

 By using the Theorem \ref{homogeneous-Jordan-block-F101}, the dimension of $\mathcal{L}_{r,s}(0,0,O_{r,s})$ is $\min\{r,s\}$, which proves that the equalities from hypothesis are sufficient conditions for the existence of solutions for \eqref{Sylvester-Jordan-block-particular}.
 
\end{proof}

\begin{thm}\label{Hermitian-nu-este-zero}
Suppose that $C$ is a Hermitian matrix, $r=s\geq 2$ and $\nu=0$. The matrix equation \eqref{Sylvester-Jordan-block-particular} has a Hermitian solution if and only if we have the equalities:
\begin{align*}
& \sum\limits_{j=1}^{\alpha }(-1)^j\emph{Im}(c_{j,2\alpha+1-j})=0,\,\forall\alpha\in \left\{1,\dots, \left[\frac{r}{2}\right]\right\},\\
& c_{11}=0,\,\,\,(-1)^{\alpha}c_{\alpha\alpha}+2\sum\limits_{j=1}^{\alpha-1 }(-1)^j\emph{Re}(c_{j,2\alpha-j})=0,\,\forall\alpha\in \left\{2,\dots, \left[\frac{r+1}{2}\right]\right\}.
\end{align*}
If the above conditions are satisfied, then a particular Hermitian solution of \eqref{Sylvester-Jordan-block-particular} is $\widehat{X}^h=\frac{1}{2}(\widehat{X}+\widehat{X}^H)$, where $\widehat{X}$ is the particular solution presented in Theorem \ref{Sylvester-Jordan-block-homogeneous-thm}.
 
\end{thm}

\begin{proof} 
By using Proposition \ref{Hermitian-solutions-abc} we have that $\mathcal{L}^H_{J_r(0),C}\neq \emptyset$ if and only if $\mathcal{L}_{J_r(0),C}\neq \emptyset$. We rewrite the conditions of the Theorem \ref{Sylvester-Jordan-block-homogeneous-thm}.

For $i=2\alpha+1\in 2\mathbb{N}+1$ we obtain 
\begin{align*}
&\sum\limits_{j=1}^{2\alpha }(-1)^jc_{2\alpha+1-j,j} =  \sum\limits_{j=1}^{\alpha }(-1)^jc_{2\alpha+1-j,j}+\sum\limits_{j=\alpha+1}^{2\alpha }(-1)^jc_{2\alpha+1-j,j} \\
 & =  \sum\limits_{j=1}^{\alpha }(-1)^jc_{2\alpha+1-j,j}+  \sum\limits_{q=1}^{\alpha }(-1)^{2\alpha -q+1}{c_{q,2\alpha+1-q}} \\
& =  \sum\limits_{j=1}^{\alpha }(-1)^j\left(\overline{c_{j,2\alpha+1-j}}-{c_{j,2\alpha+1-j}}\right)
\end{align*} 
and for $i=2\alpha\in 2\mathbb{N}$ we have
\begin{align*}
& \sum\limits_{j=1}^{2\alpha-1 }(-1)^jc_{2\alpha-j,j} =  \sum\limits_{j=1}^{\alpha-1 }(-1)^jc_{2\alpha-j,j} +(-1)^{\alpha}c_{\alpha\alpha}+\sum\limits_{j=\alpha+1}^{2\alpha-1 }(-1)^jc_{2\alpha-j,j} \\
& =  \sum\limits_{j=1}^{\alpha-1 }(-1)^jc_{2\alpha-j,j} +(-1)^{\alpha}c_{\alpha\alpha}+\sum\limits_{q=1}^{\alpha -1}(-1)^{2\alpha -q}c_{q,2\alpha-q} \\
& =  (-1)^{\alpha}c_{\alpha\alpha}+ \sum\limits_{j=1}^{\alpha-1 }(-1)^j\left(\overline{c_{j,2\alpha-j}} +{c_{j,2\alpha-j}}\right).  \qed
\end{align*}
\end{proof}

\begin{thm}\label{symmetric-nu-este-zero}
Suppose that $C$ is a real symmetric matrix, $r=s\geq 2$ and $\nu=0$. The matrix equation \eqref{Sylvester-Jordan-block-particular} has a real symmetric solution if and only if for all $\alpha\in \left\{1,\dots, \left[\frac{r+1}{2}\right]\right\}$, we have the equalities:
$$c_{11}=0,\,\,\,(-1)^{\alpha}c_{\alpha\alpha}+2\sum\limits_{j=1}^{\alpha-1 }(-1)^jc_{j,2\alpha-j}=0,\,\forall\alpha\in \left\{2,\dots, \left[\frac{r+1}{2}\right]\right\}.$$
If the above conditions are satisfied, then a real symmetric solution of \eqref{Sylvester-Jordan-block-particular} is $\widehat{X}^s=\frac{1}{2}(\widehat{X}+\widehat{X}^T)$, where $\widehat{X}$ is the solution presented in Theorem \ref{Sylvester-Jordan-block-homogeneous-thm}.
 
\end{thm}

\begin{proof}
The first conditions of the above theorem are easily verified. We observe that the particular solution $\widehat{X}$ from Theorem \ref{Sylvester-Jordan-block-homogeneous-thm} is a real matrix.

\end{proof}

To write a general solution we use Theorem \ref{homogeneous-Jordan-block-F101}  and Theorem \ref{Sylvester-Jordan-block-homogeneous-thm}. For a Hermitian solution we use Proposition \ref{Hermitian-solutions-abc}, Corollary \ref{Lyapunov-general-invertible-101-hermitian}, and Theorem \ref{Hermitian-nu-este-zero}. For a real symmetric solution we use  Proposition \ref{symmetric-solutions-abc}, Corollary \ref{Lyapunov-general-invertible-1012-symmetric}, and Theorem \ref{symmetric-nu-este-zero}.
\begin{exmp}
The case $r=s=3$. The necessary and sufficient conditions for compatibility are:
$c_{11}=0,\,\,\,c_{21}-c_{12}=0,\,\,\,c_{31}-c_{22}+c_{13}=0.$
The particular and general solutions are
$$\widehat{X}=\begin{pmatrix}
-c_{12} & -c_{13} & 0 \\
c_{13}-c_{22} & -c_{23} & 0 \\
-c_{32}+c_{23} & -c_{33} & 0
\end{pmatrix},\,X=\begin{pmatrix}
-c_{12} & -c_{13} & p_1 \\
c_{13}-c_{22} & -c_{23}-p_1 & p_2 \\
-c_{32}+c_{23}+p_1 & -c_{33}-p_2 & p_3
\end{pmatrix},p_1,p_2,p_3\in \mathbb{C}.$$

If $C$ is a Hermitian matrix, then the necessary and sufficient conditions for compatibility are $c_{11}=0$, $\text{Im}(c_{12})=0$, and $2\text{Re}(c_{13})-c_{22}=0.$ The Hermitian solutions are
\begin{align*}
X= & \begin{pmatrix}
-c_{12} & -c_{13} & -\text{Im}(c_{23})+u_1 \\
-\overline{c_{13}} & -\text{Re}(c_{23})-u_1 & -\frac{c_{33}}{2}+u_2\mathfrak{i} \\
\text{Im}(c_{23})+u_1 & -\frac{c_{33}}{2}-u_2 \mathfrak{i} & u_3
\end{pmatrix},\,u_1,u_2,u_3\in \R.
\end{align*}

If $C$ is a real symmetric matrix, then the necessary and sufficient conditions for compatibility are
$c_{11}=0$, $2c_{13}-c_{22}=0.$
The general real symmetric solution is
$$X=\begin{pmatrix}
-c_{12} & -c_{13} & u_1 \\
-c_{13} & -c_{23}-u_1 & -\frac{c_{33}}{2} \\
u_1 & -\frac{c_{33}}{2} & u_3
\end{pmatrix},\,\,\,u_1,u_3\in \R.$$

\end{exmp}

\section{Lyapunov matrix equation with Jordan matrix}\label{Lyapunov matrix equation with Jordan matrix}

We suppose that the matrix $A\in \mathcal{M}_r(\mathbb{C})$ is a Jordan matrix; i.e.
\begin{equation}\label{Jordan-matrix-form-11}
A=J_{r_1}(\lambda_1)\oplus J_{r_2}(\lambda_2)\oplus\dots\oplus J_{r_k}(\lambda_k),
\end{equation}
with $k\in \mathbb{N}^*$, $\lambda_1,\dots,\lambda_k\in \mathbb{C}$, for all $\alpha\in\{1,\dots,k\}$ the matrix $J_{r_{\alpha}}(\lambda_{\alpha})\in \mathcal{M}_{r_{\alpha}}(\mathbb{C})$ is a Jordan block matrix, and $\sum\limits_{\alpha=1}^kr_{\alpha}=r$.

To describe the Lyapunov matrix equation we denote $X=(X_{\alpha\beta})_{\alpha,\beta\in \{1,\dots,k\}}$ and $C=(C_{\alpha\beta})_{\alpha,\beta\in \{1,\dots,k\}}$
with $X_{\alpha\beta},C_{\alpha\beta}\in \mathcal{M}_{r_{\alpha},r_{\beta}}(\mathbb{C})$ for all $\alpha,\beta\in \{1,\dots,k\}$.

The Lyapunov matrix equation \eqref{Lyapunov-matrix-equation} is reduced to the system
\begin{equation}\label{Lyapunov-Jordan-equivalent-system}
J_{r_{\alpha}}^H(\lambda_{\alpha})X_{\alpha\beta}+X_{\alpha\beta}J_{r_{\beta}}(\lambda_{\beta})+C_{\alpha\beta}=O_{r_{\alpha}r_{\beta}},\,\,\,\forall \alpha,\beta\in \{1,\dots,k\}.
\end{equation}
We have $k^2$ Sylvester-Lyapunov matrix equations with Jordan block matrices. 
We make the notation:
$$\mathcal{K}_A=\{(\alpha,\beta)\in \{1,\dots,k\}^2\,|\,,\lambda_{\alpha}+\overline{\lambda_{\beta}}=0\}.$$
\begin{rem}
We have $(\alpha,\alpha)\in \mathcal{K}_A$ if and only if $\alpha=a\mathfrak{i}$ with $a\in \R$. If $(\alpha,\beta)\in \mathcal{K}_A$, then $(\beta,\alpha)\in \mathcal{K}_A$.
\end{rem}

\subsection{The homogeneous case}

In this section we suppose that $C=O_r$. The system \eqref{Lyapunov-Jordan-equivalent-system} becomes
\begin{equation}\label{Lyapunov-Jordan-equivalent-system-homogeneous}
J_{r_{\alpha}}^H(\lambda_{\alpha})X_{\alpha\beta}+X_{\alpha\beta}J_{r_{\beta}}(\lambda_{\beta})=O_{r_{\alpha}r_{\beta}},\,\,\,\forall \alpha,\beta\in \{1,\dots,k\}.
\end{equation}

It is easy to obtain the following result.

\begin{prop}\label{Jordan-homogeneous-alpha-beta}
 Let $(\alpha,\beta)\in \{1,\dots,k\}^2$. Then
 $X\in\mathcal{L}(\lambda_{\alpha},\lambda_{\beta},O_{r_{\alpha}r_{\beta}})$ if and only if $X^H\in\mathcal{L}(\lambda_{\beta},\lambda_{\alpha},O_{r_{\beta}r_{\alpha}}).$
\end{prop}
\noindent For solving the system \eqref{Lyapunov-Jordan-equivalent-system-homogeneous} we solve $\frac{k(k+1)}{2}$ homogeneous matrix equations with Jordan block matrices; more precisely the equations $(\alpha,\beta)$ with $\alpha\leq \beta$.

By using the results of Section \ref{section-Jordan block matrices} we have the following theorems.
\begin{thm}\label{homogeneous-numbers-103}
Let $A$ be a Jordan matrix of the form \eqref{Jordan-matrix-form-11}.

(i) If $\mathcal{K}_A=\emptyset$, then $\mathcal{L}_{A,O_r}=\{O_r\}$. 

(ii) If $\mathcal{K}_A\neq\emptyset$, then $\dim_{\mathbb{C}}\mathcal{L}_{A,O_r}=\sum\limits_{(\alpha,\beta)\in \mathcal{K}_A}\min\{r_{\alpha},r_{\beta}\}$.  
\end{thm}

\begin{thm}[\bf Hermitian solutions]\label{Hermitian-Jordan-1001}
Let $A$ be of the form \eqref{Jordan-matrix-form-11}.

(i) If $\mathcal{K}_A=\emptyset$, then $\mathcal{L}^H_{A,O_r}=\{O_r\}$. 

(ii) If $\mathcal{K}_A\neq\emptyset$, then $\dim_{\mathbb{R}}\mathcal{L}^H_{A,O_r}=\sum\limits_{(\alpha,\beta)\in \mathcal{K}_A}\min\{r_{\alpha},r_{\beta}\}$. 
\end{thm}

\begin{thm}\label{Jordan-matrix-positive-definite-10}
Let $A$ be a Jordan matrix. The homogeneous Lyapunov matrix equation \eqref{Lyapunov-matrix-equation-zero} has a  Hermitian positive definite solution if and only $A$ is a diagonal matrix with purely imaginary numbers on the diagonal  (all the eigenvalues of $A$ are purely imaginary and semisimple). 

\end{thm}

\begin{proof}
Suppose that $A$ has the form \eqref{Jordan-matrix-form-11} and all the eigenvalues of $A$ are purely imaginary and semisimple, then $k=r$ and $r_1=\dots=r_k=1$. $I_r$ is a Hermitian positive definite solution of \eqref{Lyapunov-matrix-equation-zero}. 

Suppose that the homogeneous Lyapunov matrix equation \eqref{Lyapunov-matrix-equation-zero} has a  Hermitian positive definite solution $X$. Let $i\in \{1,\dots,k\}$, then $X_{ii}$ is a Hermitian positive definite matrix which is the solution of the homogeneous Lyapunov matrix equation $J^H_{r_i}(\lambda_i)X_{ii}+X_{ii}J_{r_i}(\lambda_i)=O_{r_i}$. This matrix equation is equivalent with $J^H_{r_i}(\lambda_i+\overline{\lambda_i})X_{ii}+X_{ii}J_{r_i}(0)=O_{r_i}$. By using Theorem \ref{homogeneous-Jordan-block-F101} and Corollary \ref{Lyapunov-general-invertible-101-hermitian} we obtain that $r_i=1$ and  $\lambda_i+\overline{\lambda_i}=0$, which implies the result. \qed

\end{proof}

\begin{thm} [\bf Real symmetric solutions]\label{Real-symmetric-solution-Jordan-232}
Let $A$ be a real matrix \eqref{Jordan-matrix-form-11}.

(i) If $\mathcal{K}_A=\emptyset$, then $\mathcal{L}^{sym}_{A,O_r}=\{O_r\}$.  

(ii) If $\mathcal{K}_A\neq\emptyset$, then $\dim_{\mathbb{R}}\mathcal{L}^{sym}_{A,O_r}=\sum\limits_{(\alpha,\alpha)\in \mathcal{K}_A}\left[\frac{r_{\alpha}+1}{2}\right]+\sum\limits_{\stackrel{(\alpha,\beta)\in \mathcal{K}_A}{\alpha< \beta}}\min\{r_{\alpha},r_{\beta}\}$. 

\end{thm}

\begin{exmp}
If $A$ is a diagonal matrix and $\mathcal{K}_A\neq \emptyset$, then
\begin{align*}
\mathcal{L}_{A,O_r}= & \left\{\sum\limits_{(\alpha,\beta)\in \mathcal{K}_A}z_{\alpha\beta}{\bf e}_{\alpha}\otimes {\bf e}_{\beta}^T\,|\,z_{\alpha\beta}\in \mathbb{C}\right\};\footnote{${\bf e}_1,\dots,{\bf e}_r$ is the canonical base.}\\
\mathcal{L}^H_{A,O_r}= & \left\{\sum\limits_{(\alpha,\alpha)\in \mathcal{K}_A}u_{\alpha\alpha}{\bf e}_{\alpha}\otimes {\bf e}_{\alpha}^T \right.\\
+ &  \left.\sum\limits_{\stackrel{(\alpha,\beta)\in \mathcal{K}_A}{\alpha< \beta}}\left((u_{\alpha\beta}+\mathfrak{i}v_{\alpha\beta}){\bf e}_{\alpha}\otimes {\bf e}_{\beta}^T+(u_{\alpha\beta}-\mathfrak{i}v_{\alpha\beta}){\bf e}_{\beta}\otimes {\bf e}_{\alpha}^T\right)\,|\,u_{\alpha\beta},v_{\alpha\beta}\in \R\right\}; \\
\mathcal{L}^{sym}_{A,O_r}= & \left\{\sum\limits_{(\alpha,\alpha)\in \mathcal{K}_A}u_{\alpha\alpha}{\bf e}_{\alpha}\otimes {\bf e}_{\alpha}^T+\sum\limits_{\stackrel{(\alpha,\beta)\in \mathcal{K}_A}{\alpha< \beta}}u_{\alpha\beta}\left({\bf e}_{\alpha}\otimes {\bf e}_{\beta}^T+{\bf e}_{\beta}\otimes {\bf e}_{\alpha}^T\right)\,|\,u_{\alpha\beta}\in \R\right\}.
\end{align*}

\end{exmp}

\begin{exmp}\label{example-homogeneous-Jorda-block-11}
We consider the homogeneous Lyapunov matrix equation \eqref{Lyapunov-matrix-equation-zero} with Jordan matrix
$A=J_1(0)\oplus J_2(0)\oplus J_1(2)$.
In this case we have $n=4$, $k=3$, $\lambda_1=\lambda_2=0$, $\lambda_3=2$, $r_1=1$, $r_2=2$, $r_3=1$, and $\mathcal{K}_A=\{(1,1),(1,2),(2,1),(2,2)\}$. We make the notation $X=(X_{ij})_{i,j\in \{1,2,3\}}$, where 
$X_{11}, X_{13}, X_{3,1}, X_{3,3}\in \mathcal{M}_{1,1}(\mathbb{C})$, $X_{1,2}, X_{3,2}\in \mathcal{M}_{1,2}(\mathbb{C})$, $X_{2,1}, X_{2,3}\in \mathcal{M}_{2,1}(\mathbb{C})$, and $X_{2,2}\in \mathcal{M}_{2,2}(\mathbb{C})$.

The system \eqref{Lyapunov-Jordan-equivalent-system-homogeneous} becomes
\begin{equation}
\begin{cases}
J_1(0)^HX_{11}+X_{11}J_1(0)=O_{1,1} \\
J_1(0)^HX_{12}+X_{12}J_2(0)=O_{1,2} \\
J_1(0)^HX_{13}+X_{13}J_1(2)=O_{1,1} \\
J_2(0)^HX_{21}+X_{21}J_1(0)=O_{2,1} \\
J_2(0)^HX_{22}+X_{22}J_2(0)=O_{2,2} \\
J_2(0)^HX_{23}+X_{23}J_1(2)=O_{2,1} \\
J_1(2)^HX_{31}+X_{31}J_1(0)=O_{1,1} \\
J_1(2)^HX_{32}+X_{32}J_2(0)=O_{2,1} \\
J_1(2)^HX_{33}+X_{33}J_1(2)=O_{1,1}

\end{cases}.
\end{equation}
By using the results of the above section we have: $X_{11}=(x)$, with $x\in \mathbb{C}$, $X_{12}=
 \begin{pmatrix}
 0 & y 
 \end{pmatrix}$, with $y\in \mathbb{C}$, $X_{13}=O_{1,1}$, $X_{21}=
 \begin{pmatrix}
 0 \\
 z
 \end{pmatrix}$, with $z\in \mathbb{C}$, $X_{22}=
 \begin{pmatrix}
 0 & t \\
 -t & u
 \end{pmatrix}$, with $t,u\in \mathbb{C}$, $X_{23}=O_{2,1}$, $X_{31}=O_{1,1}$, $X_{32}=O_{1,2}$, and $X_{33}=O_{1,1}$.
 The general solution of the homogeneous Lyapunov matrix equation is
 $$X=\begin{pmatrix}
 x & 0 & y & 0 \\
 0 & 0 & t & 0 \\
 z & -t & u & 0 \\
 0 & 0 & 0 & 0 
 \end{pmatrix},\,\,\,x,y,z,t,u\in \mathbb{C}.$$

We have the following description of a general Hermitian solution:
$$X=\begin{pmatrix}
 x & 0 & y+z\mathfrak{i} & 0 \\
 0 & 0 & t\mathfrak{i} & 0 \\
 y-z\mathfrak{i}& -t\mathfrak{i} & u & 0 \\
 0 & 0 & 0 & 0 
 \end{pmatrix},\,\,\,x,y,z,t,u\in \mathbb{R}.$$

A general real symmetric solution has the form:
$$X=\begin{pmatrix}
 x & 0 & y & 0 \\
 0 & 0 & 0 & 0 \\
 y & 0 & u & 0 \\
 0 & 0 & 0 & 0 
 \end{pmatrix},\,\,\,x,y,u\in \mathbb{R}.$$
\end{exmp}
\medskip 

In what follows we study the invertible solutions.

\begin{lem}\label{invertible-1122}
Let $r_1,\dots,r_k\in \mathbb{N}^*$, $r:=\sum_{i=1}^kr_i$, $A_1,\in \mathcal{M}_{r_1}(\mathbb{C})$, $\dots, A_k\in \mathcal{M}_{r_k}(\mathbb{C})$, and $A=A_1\oplus A_2\oplus\dots\oplus A_k$. If for all $i\in\{1,\dots,k\}$ the invertible matrix $X_i$ verify  
$A_i^HX_i+X_iA_i=O_{r_i}$, then $X=X_1\oplus X_2\oplus\dots\oplus X_k$ is an invertible solution of the homogeneous Lyapunov matrix equation \eqref{Lyapunov-matrix-equation-zero}.
\end{lem}

\begin{proof}
It is easy to observe that $X$ is a solution of \eqref{Lyapunov-matrix-equation-zero}.  We have the equality 
$\det X=\prod_{i=1}^k\det X_i$ which implies the fact that $X$ is an invertible matrix.
\qed
\end{proof}

\begin{lem}\label{A-Jordan-oplus}
Let $\lambda\in \mathbb{C}$ with $\emph{Re}(\lambda)=0$, $r_1,\dots,r_k\in \mathbb{N}^*$, $r:=\sum_{i=1}^kr_i$, and $A:=J_{r_1}(\lambda)\oplus J_{r_2}(\lambda)\oplus\dots\oplus J_{r_k}(\lambda)$.
The homogeneous Lyapunov matrix equation \eqref{Lyapunov-matrix-equation-zero} has an invertible solution.
\end{lem}

\begin{proof}
We observe that the matrix
$X=\mathcal{Y}^{[r_1r_1]}\oplus \mathcal{Y}^{[r_2r_2]}\oplus \dots \oplus \mathcal{Y}^{[r_kr_k]}$
is an invertible solution, where the matrix $\mathcal{Y}^{[r_ir_i]}$ is defined in Section \ref{appendix}. \qed
\end{proof}

\begin{lem}\label{invertible-1133}
Let $\lambda,\mu\in \mathbb{C}$ with $\lambda\neq \mu$ and $\overline{\lambda}+\mu=0$, $r_1,\dots,r_k\in \mathbb{N}^*$, $r:=\sum_{i=1}^kr_i$, $A(\lambda):=J_{r_1}(\lambda)\oplus J_{r_2}(\lambda)\oplus\dots\oplus J_{r_k}(\lambda)$, $A(\mu):=J_{r_1}(\mu)\oplus J_{r_2}(\mu)\oplus\dots\oplus J_{r_k}(\mu)$, and $A:=A(\lambda)\oplus A(\mu)$.
The matrix equation \eqref{Lyapunov-matrix-equation-zero} has an invertible solution.
\end{lem}

\begin{proof}
We denote $X=\begin{pmatrix}
X_{11} & X_{12} \\
X_{21} & X_{22}
\end{pmatrix}$; it follows that \eqref{Lyapunov-matrix-equation-zero} is equivalent with the system 
$$\begin{cases}
A^H(\lambda)X_{11}+X_{11}A(\lambda)=O_r \\
A^H(\lambda)X_{12}+X_{12}A(\mu)=O_r \\
A^H(\mu)X_{21}+X_{21}A(\lambda)=O_r \\
A^H(\mu)X_{22}+X_{22}A(\mu)=O_r.
\end{cases}$$

We observe that the first matrix equation is equivalent with $A^H(\lambda+\overline{\lambda})X_{11}+X_{11}A(0)=O_r$. Because $\lambda+\overline{\lambda}\neq 0$ we deduce that $X_{11}=O_r$. Analogously it can be proved that $X_{22}=O_r$.

The second matrix equation is equivalent with $A^H(0)X_{12}+X_{12}A(0)=O_r$. 
By using Lemma \ref{A-Jordan-oplus} we obtain $\widehat{X}_{12}$ and $\widehat{X}_{21}=\widehat{X}_{12}$.
The particular solution $\widehat{X}=\begin{pmatrix}
O_r & \widehat{X}_{12} \\
\widehat{X}_{21} & O_r
\end{pmatrix}$ is invertible because $\det \widehat{X}=\pm\det \widehat{X}_{12}\cdot\det \widehat{X}_{21}\neq 0$. \qed
\end{proof}

\begin{thm}\label{Jordan-invertible-545}
Suppose that $A$ is a Jordan matrix of the form \eqref{Jordan-matrix-form-11}.
The necessary and sufficient condition for the homogeneous Lyapunov matrix equation \eqref{Lyapunov-matrix-equation-zero} to have an invertible solution is that for all $ \alpha\in \{1,\dots,k\}$ we have $-\overline{\lambda_{\alpha}}\in  \{\lambda_1,\dots,\lambda_k\}$ and ${\bf w}(A,\lambda_{\alpha})={\bf w}(A,-\overline{\lambda_{\alpha}})$.\footnote{${\bf w}$ is the Weyr characteristic, see Section \ref{appendix}.}
\end{thm}

\begin{proof}
If the homogeneous Lyapunov matrix equation \eqref{Lyapunov-matrix-equation-zero} has an invertible solution, then the matrices $A^H$ and $-A$ are similar. We deduce that $\forall \alpha\in \{1,\dots,k\}\Rightarrow -\overline{\lambda_{\alpha}}\in  \{\lambda_1,\dots,\lambda_k\}$ and $\forall \alpha\in \{1,\dots,k\}$ we have ${\bf w}(A^H,{\lambda_{\alpha}})={\bf w}(-A,-{\lambda_{\alpha}})$ which implies ${\bf w}(A,\lambda_{\alpha})={\bf w}(A,-\overline{\lambda_{\alpha}})$.

The reverse problem. We denote $\mu_1,\dots, \mu_q$ ($q\leq k$) the distinct elements of the set $\{\lambda_1,\dots,\lambda_k\}$ such that: $\mu_{\gamma}+\overline{\mu_{\gamma}}\neq 0$ for ${\gamma}\in \{1,\dots, 2p\}$, $\mu_{\gamma}+\overline{\mu_{\gamma}}=0$ for $\gamma\in \{2p+1,\dots, q\}$,  $\mu_{2\gamma}=-\overline{\mu_{2\gamma-1}}$ for $\gamma\in \{1,\dots, p\}$. We construct the matrix
$B:=B(\mu_1)\oplus B(\mu_2)\oplus\dots\oplus B(\mu_q)$,
where $B(\mu_i)$ contains all the Jordan blocks of $A$ with $\mu_i$ on the diagonal; more precisely we have 
$B(\mu_i)= J_{t_1^i}(\mu_i)\oplus J_{t_2^i}(\mu_i)\oplus\dots\oplus J_{t_{w_1(A,\lambda)}^i}(\mu_i).$ By hypotheses for $i\in \{1,\dots, p\}$ we have $w_1(A,\mu_{2i-1})=w_1(A,\mu_{2i})$, $t_1^{2i-1}=t_1^{2i}$, ..., $t_{w_1(A,\mu_{2i-1})}^{2i-1}=t_{w_1(A,\mu_{2i})}^{2i}$.

For $i\in \{1,\dots, p\}$ we note $B(\mu_{2i-1},\mu_{2i})=B(\mu_{2i-1})\oplus B(\mu_{2i})$ and we write 
$$B=B(\mu_{1},\mu_{2})\oplus\dots \oplus B(\mu_{2p-1},\mu_{2p}) \oplus B(\mu_{2p+1})\oplus \dots\oplus B(\mu_q).$$

By using Lemma 3.1.18 from \cite{horn} we deduce that $A$ and $B$ are similar matrices; $B=P^{-1}AP$.
We consider the homogeneous Lyapunov matrix equation 
$$B^HY+YB=O_r.$$
By using Lemma \ref{invertible-1122}, Lemma \ref{A-Jordan-oplus}, and Lemma \ref{invertible-1133} the above equation has an invertible solution $\widehat{Y}$. From Proposition \ref{equivalence-Jordan-102} we deduce that $\widehat{X}=P^{-H}\widehat{Y}P^{-1}$ is an invertible solution of homogeneous Lyapunov matrix equation  \eqref{Lyapunov-matrix-equation-zero}. \qed
\end{proof}

\begin{rem}
In the above results we have a method to construct an invertible solution if the conditions of Theorem \ref{Jordan-invertible-545} are satisfied.

\end{rem}

\begin{cor} 

(i) If $A$ is a Jordan matrix of the form \eqref{Jordan-matrix-form-11} and for all $\alpha \in \{1,\dots, k\}$ the number $\lambda_{\alpha}$ is purely imaginary, then the homogeneous Lyapunov matrix equation \eqref{Lyapunov-matrix-equation-zero} has an invertible solution.

(ii) If $A$ is a nilpotent Jordan matrix, then the homogeneous Lyapunov matrix equation \eqref{Lyapunov-matrix-equation-zero} has an invertible solution.
\end{cor}

\subsection{The non-homogeneous case}

\begin{thm}
Let $A$ be a Jordan matrix of the form \eqref{Jordan-matrix-form-11}.

(i) If $\mathcal{K}_A=\emptyset$, then the Lyapunov matrix equation \eqref{Lyapunov-matrix-equation} has a unique solution.  

(ii) If $\mathcal{K}_A\neq\emptyset$, then the matrix equation has solutions if and only if the components of the matrices $C_{\alpha\beta}$, with $(\alpha,\beta)\in \mathcal{K}_A$, verify the conditions $\sum\limits_{j=1}^{i-1}(-1)^jc^{(\alpha\beta)}_{i-j,j}=0$, $\forall i\in \{2,\dots, \min\{r_{\alpha},r_{\beta}\}+1\}$.
\end{thm}

\begin{exmp}
Suppose that $A$ is a diagonal matrix, i.e. $A=diag(\lambda_1,\dots,\lambda_r)$. 

\noindent If $\mathcal{K}_A=\emptyset$, then $\mathcal{L}_{A,C}=\left\{-\sum\limits_{\alpha,\beta=1}^r\dfrac{c_{\alpha\beta}}{\overline{\lambda_{\alpha}}+\lambda_{\beta}}{\bf e}_{\alpha}\otimes{\bf e}_{\beta}^T\right\}$. 

\noindent If $\mathcal{K}_A\neq \emptyset$, then the compatibility conditions are $c_{\alpha\beta}=0$ for all $(\alpha,\beta)\in \mathcal{K}$. We have
$$\mathcal{L}_{A,C}=\left\{-\sum\limits_{(\alpha,\beta)\notin \mathcal{K}_A}\frac{c_{\alpha\beta}}{\overline{\lambda_{\alpha}}+\lambda_{\beta}}{\bf e}_{\alpha}\otimes{\bf e}_{\beta}^T-\sum\limits_{(\alpha,\beta)\in \mathcal{K}_A}z_{\alpha\beta}{\bf e}_{\alpha}\otimes{\bf e}_{\beta}^T \,|\,z_{\alpha\beta}\in \mathbb{C}\right\}.$$

\end{exmp}

\begin{exmp}
We consider the matrix equation \eqref{Lyapunov-matrix-equation} with the Jordan matrix $A$ as in Example \ref{example-homogeneous-Jorda-block-11}. We use the notations presented in Example \ref{example-homogeneous-Jorda-block-11} and we denote by
$C=\begin{pmatrix}
C_{11} & C_{12} & C_{13} \\
C_{21} & C_{22}  & C_{23}\\
C_{31} & C_{32} & C_{33}
\end{pmatrix}$, where 
\begin{align*}
C_{11} & =\begin{pmatrix}
c_{11}
\end{pmatrix},\,\,
C_{12}=\begin{pmatrix}
c_{12} & c_{13}
\end{pmatrix},\,\,C_{13}=\begin{pmatrix}
c_{14}
\end{pmatrix}, \\
C_{21} & =\begin{pmatrix}
c_{21} \\
c_{31}
\end{pmatrix},\,\,
C_{22}=\begin{pmatrix}
c_{22} & c_{23} \\
c_{32} & c_{33}
\end{pmatrix},\,\,C_{23}=\begin{pmatrix}
c_{24} \\
c_{34}
\end{pmatrix}, \\
C_{31} & =\begin{pmatrix}
c_{41}
\end{pmatrix},\,\,
C_{32}=\begin{pmatrix}
c_{42} & c_{43}
\end{pmatrix},\,\,C_{33}=\begin{pmatrix}
c_{44}
\end{pmatrix}.
\end{align*}
The system \eqref{Lyapunov-Jordan-equivalent-system} becomes
\begin{equation}
\begin{cases}
J_1(0)^HX_{11}+X_{11}J_1(0)+C_{11}=O_{1,1} \\
J_1(0)^HX_{12}+X_{12}J_2(0)+C_{12}=O_{1,2} \\
J_1(0)^HX_{13}+X_{13}J_1(2)+C_{13}=O_{1,1} \\
J_2(0)^HX_{21}+X_{21}J_1(0)+C_{21}=O_{2,1} \\
J_2(0)^HX_{22}+X_{22}J_2(0)+C_{22}=O_{2,2} \\
J_2(0)^HX_{23}+X_{23}J_1(2)+C_{23}=O_{2,1} \\
J_1(2)^HX_{31}+X_{31}J_1(0)+C_{31}=O_{1,1} \\
J_1(2)^HX_{32}+X_{32}J_2(0)+C_{32}=O_{2,1} \\
J_1(2)^HX_{33}+X_{33}J_1(2)+C_{33}=O_{1,1}.
\end{cases}
\end{equation}
The compatibility conditions are
$c_{11}=c_{12}=c_{21}=c_{22}=0,\,\,c_{23}-c_{32}=0.$

If the compatibility conditions are satisfied, then the solution has the form
$$X=\begin{pmatrix}
p & -c_{13} & q & -\frac{c_{14}}{2} \\
-c_{31} & -c_{23} & v & -\frac{c_{24}}{2} \\
u & -c_{33}-v & w & \frac{c_{24}}{4}-\frac{c_{34}}{2} \\
-\frac{c_{41}}{2} & -\frac{c_{42}}{2} & \frac{c_{42}}{4}-\frac{c_{43}}{2} & -\frac{c_{44}}{4}
\end{pmatrix},\,\,\,p,q,u,v,w\in \mathbb{C}.$$

If $C$ is a Hermitian matrix, then a Hermitian solution has the form
$$X=\begin{pmatrix}
p_1 & -c_{13} & q_1+q_2 i & -\frac{c_{14}}{2} \\
-c_{31} & -c_{23} & -\frac{c_{33}}{2}+v_2 i & -\frac{c_{24}}{2} \\
q_1-q_2 i & -\frac{c_{33}}{2}-v_2 i & w_1 & \frac{c_{24}}{4}-\frac{c_{34}}{2} \\
-\frac{c_{41}}{2} & -\frac{c_{42}}{2} & \frac{c_{42}}{4}-\frac{c_{43}}{2} & -\frac{c_{44}}{4}
\end{pmatrix},\,\,\,p_1,q_1,q_2,v_2,w_1\in \R.$$

If $C$ is a real symmetric matrix, then a real symmetric solution has the form
$$X=\begin{pmatrix}
p_1 & -c_{13} & q_1 & -\frac{c_{14}}{2} \\
-c_{31} & -c_{23} & -\frac{c_{33}}{2} & -\frac{c_{24}}{2} \\
q_1 & -\frac{c_{33}}{2} & w_1 & \frac{c_{24}}{4}-\frac{c_{34}}{2} \\
-\frac{c_{41}}{2} & -\frac{c_{42}}{2} & \frac{c_{42}}{4}-\frac{c_{43}}{2} & -\frac{c_{44}}{4}
\end{pmatrix},\,\,\,p_1,q_1,w_1\in \R.$$

\end{exmp}

\section{Lyapunov matrix equation. General case.}\label{Lyapunov matrix equation. General case.}

For an arbitrary matrix $A$ we can use the results presented in the previous sections when we can find the eigenvalues of the matrix $A$. In this case we can construct the Jordan matrix $J_A$ which is similar to the matrix $A$ and we find some theoretical results.   To find the solutions of the Lyapunov matrix equation in the general case we need to know the invertible matrix $P_A$ which appears in the condition for similarity of the matrices $A$ and $J_A$.   

We consider the Lyapunov matrix equation \eqref{Lyapunov-matrix-equation} and we denote by $J_A\in \mathcal{M}_r(\mathbb{C})$ a Jordan matrix such that $A$ and $J_A$ are similar matrices.  The invertible matrix $P_A\in \mathcal{M}_r(\mathbb{C})$ verifies the equality $J_A=P_A^{-1}AP_A$. For a real matrix $A$ with all the eigenvalues being real numbers we choose $P_A$ to be a real matrix. We have $J_A=J_{r_1}(\lambda_1)\oplus J_{r_2}(\lambda_2)\oplus \dots\oplus J_{r_k}(\lambda_k)$ 
with $k\in \mathbb{N}^*$, $\lambda_1,\dots,\lambda_k\in \mathbb{C}$, for all $\alpha\in \{1,\dots ,k\}$ the matrix $J_{r_{\alpha}}(\lambda_{\alpha})\in \mathcal{M}_{r_{\alpha}}(\mathbb{C})$ is a Jordan block matrix, and $\sum\limits_{\alpha=1}^kr_{\alpha}=r$. We denote by
$$\mathcal{K}_A=\{(\alpha,\beta)\in \{1,\dots,k\}\,|\,\lambda_{\alpha}+\overline{\lambda_{\beta}}=0\}.$$

We construct the Lyapunov matrix equation with Jordan matrix 
\begin{equation}\label{L-J-matrix-111}
J_A^HZ+ZJ_A+D=O_r,
\end{equation}
where $D=P_A^HCP_A$. A solution of the Lyapunov matrix equation \eqref{Lyapunov-matrix-equation} has the form $X=P_A^{-H}ZP_A^{-1}$, where $Z$ is a solution of the the Lyapunov matrix equation with Jordan matrix \eqref{L-J-matrix-111}.

\subsection{The homogeneous case}
In this case we have $C=O_r$ and consequently we obtain $D=O_r$. The Lyapunov matrix equation with Jordan matrix \eqref{L-J-matrix-111} becomes
\begin{equation}\label{L-J-matrix-111-hom}
J_A^HZ+ZJ_A=O_r.
\end{equation}

The following proposition is obtained by direct computations.

\begin{prop}\label{isomorphism-Jordan-12}
Consider the function $\mathcal{F}_A:\mathcal{L}_{J_A,O_r}\rightarrow \mathcal{L}_{A,O_r}$ given by $\mathcal{F}_A(Z)=P_A^{-H}ZP_A^{-1}$.

(i) $\mathcal{F}_A$ is an $\mathbb{C}$-isomorphism.

(ii) The restriction ${\mathcal{F}_{A}}_{\left|_{\mathcal{L}^H_{J_A,O_r}}\right.}:\mathcal{L}^H_{J_A,O_r}\rightarrow \mathcal{L}^H_{A,O_r}$ is an $\R$-isomorphism.

(iii) $Z\in \mathcal{L}^H_{J_A,O_r}$ is positive definite if and only if $\mathcal{F}_A(Z)\in \mathcal{L}^H_{A,O_r}$ is positive definite.

(iv) If $A,J_A$, and $P_A$ are real matrices, then the restriction ${\mathcal{F}_{A}}_{\left|_{\mathcal{L}^{sym}_{J_A,O_r}}\right.}:\mathcal{L}^{sym}_{J_A,O_r}\rightarrow \mathcal{L}^{sym}_{A,O_r}$ is an $\R$-isomorphism.
\end{prop}

 By using Proposition \ref{equivalence-Jordan-102}, Theorem \ref{homogeneous-numbers-103}, Theorem \ref{Hermitian-Jordan-1001}, and Theorem \ref{Real-symmetric-solution-Jordan-232} we obtain the following result.

\begin{thm}
The following holds true.

(i) If $\mathcal{K}_A=\emptyset$, then $ \mathcal{L}_{A,O_r}=\{O_r\}$.

(ii) If $\mathcal{K}_A\neq\emptyset$, then $\dim_{\mathbb{C}}\mathcal{L}_{A,O_r}=\sum\limits_{(\alpha,\beta)\in \mathcal{K}_A}\min\{r_{\alpha},r_{\beta}\}$.

(iii)  If $\mathcal{K}_A\neq\emptyset$, then $\dim_{\mathbb{R}}\mathcal{L}^H_{A,O_r}=\sum\limits_{(\alpha,\beta)\in \mathcal{K}_A}\min\{r_{\alpha},r_{\beta}\}$.

(iv) If $A$ is a real matrix, $\emph{Spec}(A)\subset \R$, and $\mathcal{K}_A\neq\emptyset$, then $\dim_{\mathbb{R}}\mathcal{L}^{sym}_{A,O_r}=\sum\limits_{(\alpha,\alpha)\in \mathcal{K}_A}\left[\frac{r_{\alpha}+1}{2}\right]+\sum\limits_{\stackrel{(\alpha,\beta)\in \mathcal{K}_A}{\alpha< \beta}}\min\{r_{\alpha},r_{\beta}\}$. 

\end{thm}

\begin{cor}
If $A\in \mathcal{M}_r(\mathbb{C})$, then $$\emph{rank} (I_r\otimes A^H+\overline{A^H}\otimes I_r)=r^2-\sum\limits_{(\alpha,\beta)\in \mathcal{K}_A}\min\{r_{\alpha},r_{\beta}\}.$$
\end{cor}

\begin{proof}
By using the rank-nullity theorem we have that 
$$\text{rank}(I_r\otimes A^H+\overline{A^H}\otimes I_r)+ \text{dim}_{\mathbb{C}}\{{\bf x}\in \mathbb{C}^{n^2}\,|(I_r\otimes A^H+\overline{A^H}\otimes I_r){\bf x}={\bf 0}\}=r^2.$$
The homogeneous Lyapunov matrix equation \eqref{Lyapunov-matrix-equation-zero} is equivalent with the linear system of equations 
\begin{equation}
(I_r\otimes A^H+\overline{A^H}\otimes I_r){\bf vec}(X)={\bf 0}.
\end{equation}
By using the above theorem we obtain the announced result. \qed
\end{proof}

\begin{thm} 
The necessary and sufficient condition for the homogeneous Lyapunov matrix equation \eqref{Lyapunov-matrix-equation-zero} to have an invertible solution is that for all $\lambda\in \emph{Spec}(A)$ we have $-\overline{\lambda}\in  \emph{Spec}(A)$ and ${\bf w}(A,\lambda)={\bf w}(A,-\overline{\lambda})$\footnote{${\bf w}$ is the Weyr characteristic, see Section \ref{appendix}}.
\end{thm}

\begin{proof}
We use Theorem \ref{Jordan-invertible-545}, $\text{Spec}(A)=\text{Spec}(J_A)$ and the fact that for all $\lambda\in \text{Spec}(A)$ we have ${\bf w}(A,\lambda)={\bf w}(J_A,\lambda)$. \qed
\end{proof}

\begin{thm}
The homogeneous Lyapunov matrix equation \eqref{Lyapunov-matrix-equation-zero} has a  Hermitian positive definite solution if and only if $A$ is a diagonalizable matrix with all eigenvalues being purely imaginary (i.e. the eigenvalues are purely imaginary and semisimple). 
\end{thm}

\begin{proof}
If all the eigenvalues of $A$ are purely imaginary and semisimple, then all eigenvalues of $J_A$ are purely imaginary and semisimple and, by using Theorem \ref{Jordan-matrix-positive-definite-10}, we deduce that the matrix equation \eqref{L-J-matrix-111-hom} has a Hermitian positive definite solution $Z$. The Proposition \ref{isomorphism-Jordan-12} implies that  $\mathcal{F}_A(Z)$ is a Hermitian positive definite solution of \eqref{Lyapunov-matrix-equation-zero}.

If the homogeneous Lyapunov matrix equation \eqref{Lyapunov-matrix-equation-zero} has a  Hermitian positive definite solution $X$, then $\mathcal{F}^{-1}_A(X)$ is a Hermitian positive definite solution of \eqref{L-J-matrix-111-hom}. Using Theorem \ref{Jordan-matrix-positive-definite-10} we have that all eigenvalues of $A$ are purely imaginary and semisimple. \qed
\end{proof}

\begin{exmp}
If $A$ is a diagonalizable matrix, then $J_A=\text{diag}(\lambda_1,\dots,\lambda_r)$. 

If $\mathcal{K}_A= \emptyset$, then $O_r$ is the unique solution of the homogeneous Lyapunov matrix equation \eqref{Lyapunov-matrix-equation-zero}. 

If $\mathcal{K}_A\neq \emptyset$, then 
$\mathcal{L}_{A,O_r}=\left\{\sum\limits_{(\alpha,\beta)\in \mathcal{K}_A}z_{\alpha\beta}P_A^{-H}{\bf e}_{\alpha}\otimes {\bf e}_{\beta}^TP_A^{-1}\,|\,z_{\alpha\beta}\in \mathbb{C}\right\}.$

\end{exmp}

\begin{exmp}\label{exemplul-general-1001}
We consider the homogeneous Lyapunov matrix equation with $A=\begin{pmatrix}
 0 & 0 & 0 & 0 \\
  -1 & 0 & 0 & 0 \\
  0 & 0 & 0 & 0 \\
   0 & 0 & 2 & 2
\end{pmatrix}$. The Jordan matrix is $J_A=J_1(0)\oplus J_2(0)\oplus J_1(2)$, the set $\mathcal{K}_A=\{(1,1),(1,2),(2,1),(2,2)\}$ and the matrix $P_A=\begin{pmatrix}
 0 & 0 & 1 & 0 \\
  0 & -1 & 0 & 0 \\
  -1 & 0 & 1 & 0 \\
   1 & 0 & -1 & 1
\end{pmatrix}$. By using the Example \ref{example-homogeneous-Jorda-block-11} we obtain the general  solution 
$$X=\begin{pmatrix}
 x+y+z+u & t & -x-z & 0 \\
  -t & 0 & 0 & 0 \\
  -x-y & 0 & x & 0 \\
   1 & 0 & 0 & 0
\end{pmatrix},\,\,\,
x,y,z,t,u\in \mathbb{C}.$$

\end{exmp}

\subsection{The non-homogeneous case}

We  write 
$D=(D_{\alpha\beta})_{\alpha,\beta\in \{1,\dots,k\}}
$
with $D_{\alpha\beta}=\left(d_{ij}^{(\alpha\beta)}\right)\in \mathcal{M}_{r_{\alpha},r_{\beta}}(\mathbb{C})$ for all $\alpha,\beta$.

\begin{thm}

(i) If $\mathcal{K}_A=\emptyset$, then the Lyapunov matrix equation \eqref{Lyapunov-matrix-equation} has a unique solution.  

(ii) If $\mathcal{K}_A\neq\emptyset$, then the matrix equation has solutions if and only if the components of the matrices $D_{\alpha\beta}$, with $(\alpha,\beta)\in \mathcal{K}_A$, verify the conditions $\sum\limits_{j=1}^{i-1}(-1)^jd^{(\alpha\beta)}_{i-j,j}=0$, $\forall i\in \{2,\dots, \min\{r_{\alpha},r_{\beta}\}+1\}$.
\end{thm}

\begin{rem}
In the papers \cite{roth} and \cite{lancaster-tismenetsky} it is presented a necessary and sufficient condition for the existence of solutions for a Sylvester matrix equation. In the case of the Lyapunov  matrix equation \eqref{Lyapunov-matrix-equation} this condition says that $\begin{pmatrix}
A^H & O_r \\
O_r & -A
\end{pmatrix}$ and $\begin{pmatrix}
A^H & -C \\
O_r & -A
\end{pmatrix}$ are similar matrices. We assess that the conditions presented in the previous theorem are easier to use in a practical example.
\end{rem}

\begin{exmp}
If $A$ is a diagonalizable matrix, then $J_A=\text{diag}(\lambda_1,\dots,\lambda_r)$. 

\noindent If $\mathcal{K}_A= \emptyset$, then $\mathcal{L}_{A,C}=\left\{-\sum\limits_{\alpha,\beta=1}^r\dfrac{c_{\alpha\beta}}{\overline{\lambda_{\alpha}}+\lambda_{\beta}}P_A^{-H}{\bf e}_{\alpha}\otimes{\bf e}_{\beta}^TP_A^{-1}\right\}$. 

\noindent If $\mathcal{K}_A\neq \emptyset$, then 
$\mathcal{L}_{A,C}=  \left\{-\sum\limits_{(\alpha,\beta)\notin \mathcal{K}_A}\dfrac{c_{\alpha\beta}}{\overline{\lambda_{\alpha}}+\lambda_{\beta}}P_A^{-H}{\bf e}_{\alpha}\otimes{\bf e}_{\beta}^TP_A^{-1}\right. \\
 \left.\,\,\,\,\,\,\,\,\,\,\,\,\,\,\,\,\,\,\,\,\,\,\,\,\,\,\,\,\,\,\,\,\,\,\,\,\,\,\,\,\,\,\,\,\,\,\,\,\,\,\,\,\,\,\,\,\,\,\,\,\,\,\,\,\,\,\,\,\,\,\,\,\,\,\,\,\,\,\,\,\,
 \,\,\,\,\,\,\,\,\,\,\,\,\,\,\,\,
 -\sum\limits_{(\alpha,\beta)\in \mathcal{K}_A}z_{\alpha\beta}P_A^{-H}{\bf e}_{\alpha}\otimes {\bf e}_{\beta}^TP_A^{-1}\,|\,z_{\alpha\beta}\in \mathbb{C}\right\}.$

\end{exmp}

\begin{exmp}
We consider a Lyapunov matrix equation \eqref{Lyapunov-matrix-equation} with the matrix $A$ from the Example \ref{exemplul-general-1001} and an arbitrary matrix $C$. The Jordan matrix $J_A$ and $P_A$ are the ones from the above example. The matrix $D$ is formed with the blocks:
{\small $$D_{11}=\begin{pmatrix}
c_{{3,3}}-c_{{4,3}}-c_{{3,4}}+c_{{4,4}}
\end{pmatrix};\,\, D_{12}=\left( \begin {array}{cc} c_{{3,2}}-c_{{4,2}}&-c_{{3,1}}+c_{{4,1}}-c_
{{3,3}}+c_{{4,3}}+c_{{3,4}}-c_{{4,4}}\end {array} \right),$$
$$D_{13}=\left( \begin {array}{c}-c_{3, 4}+c_{4, 4}\end {array} \right),\,\,D_{21}=\left( \begin {array}{c} c_{{2,3}}-c_{{2,4}}\\ \noalign{\medskip}-c_{
{1,3}}-c_{{3,3}}+c_{{4,3}}+c_{{1,4}}+c_{{3,4}}-c_{{4,4}}\end {array}
 \right),
$$
$$D_{22}=\left( \begin {array}{cc} c_{{2,2}}&-c_{{2,1}}-c_{{2,3}}+c_{{2,4}}
\\ \noalign{\medskip}-c_{{1,2}}-c_{{3,2}}+c_{{4,2}}&c_{{1,1}}+c_{{3,1}
}-c_{{4,1}}+c_{{1,3}}+c_{{3,3}}-c_{{4,3}}-c_{{1,4}}-c_{{3,4}}+c_{{4,4}
}\end {array} \right),
$$
$$D_{31}=\left( \begin {array}{c} -c_{{4,3}}+c_{{4,4}}\end {array} \right),\,\,D_{32}=\left( \begin {array}{cc} -c_{{4,2}}&c_{{4,1}}+c_{{4,3}}-c_{{4,4}}
\end {array} \right),\,\,D_{33}=\left( \begin {array}{c} c_{{4,4}}\end {array} \right).$$}
The compatibility conditions become
$$\begin{cases}
c_{{3,3}}-c_{{4,3}}-c_{{3,4}}+c_{{4,4}}=0 \\
c_{{3,2}}-c_{{4,2}}=0 \\
c_{{2,3}}-c_{{2,4}}=0 \\
c_{{2,2}}=0 \\
-c_{{2,1}}-c_{{2,3}}+c_{{2,4}}+c_{{1,2}}+c_{{3,2}}-c_{{4,2}}=0.
\end{cases}$$
By using Theorem \ref{homogeneous-Jordan-block-F101} and Theorem \ref{nonhomogeneous-11-Jordan-block} we obtain the following particular solution of the Lyapunov matrix equation \eqref{Lyapunov-matrix-equation}:

{\small $$\widehat{X}=\left( \begin {array}{cccc} 0&c_{{1,1}}+c_{{1,3}}-c_{{1,4}}&-\frac{1}{4}c_{
{2,3}}-\frac{1}{2}c_{{1,4}}&-\frac{1}{4}c_{{2,3}}-\frac{1}{2}c_{{1,4}}
\\ \noalign{\medskip}-c_{{1,3}}+c_{{1,4}}&c_{{1,2}}&c_{{1,3}}-c_{{1,4}
}-\frac{1}{2}c_{{2,3}}&-\frac{1}{2}c_{{2,3}}\\ \noalign{\medskip}-\frac{1}{4}c_{{3,2}}-\frac{1}{2}c_{{4,1}}&c_{{3,1}}-c_{{4,1}}-\frac{1}{2}c_{{3,2}}&-\frac{3}{4}c_{{3,3}}+\frac{1}{4}
c_{{3,4}}+\frac{1}{4}c_{{4,3}}&-\frac{1}{4}c_{{3,3}}+\frac{1}{4}c_{{4,3}}-\frac{1}{4}c_{{3,4}}
\\ \noalign{\medskip}-\frac{1}{4}c_{{3,2}}-\frac{1}{2}c_{{4,1}}&-\frac{1}{2}c_{{3,2}}&-\frac{1}{4}c_{{3,3}}+\frac{1}{4}c_{{3,4}}-\frac{1}{4}c_{{4,3}}&\frac{1}{4}c_{{3,3}}-\frac{1}{4}c_{{4,
3}}-\frac{1}{4}c_{{3,4}}\end {array} \right). $$}
By using Example \ref{exemplul-general-1001} we can write the general solution.

\end{exmp}

\begin{appendices}

\section{Matrix notions and results}\label{appendix}

\noindent $\Box$ $\mathcal{M}_{r\times s}(\mathbb{C})$ is the set of matrices with complex elements which have $r$ rows and $s$ columns. We use also the notations: $\mathcal{M}_{r}(\mathbb{C}):=\mathcal{M}_{r\times r}(\mathbb{C})$, $O_{r,s}\in \mathcal{M}_{s\times r}(\mathbb{C})$ is the null matrix, $O_r$ is the null matrix in $\mathcal{M}_{r}(\mathbb{C})$, and $I_r\in \mathcal{M}_{r}(\mathbb{C})$ is the identity matrix.
\medskip

\noindent $\Box$ For a matrix $M\in \mathcal{M}_{r\times s}(\mathbb{C})$ we make the notations: $M^T\in \mathcal{M}_{s\times r}(\mathbb{C})$ is the transpose matrix, $M^H\in \mathcal{M}_{s\times r}(\mathbb{C})$ is the conjugate transpose matrix.\medskip

\noindent $\Box$ $M\in \mathcal{M}_{r}(\mathbb{C})$ is a Hermitian matrix if $M^H=M$. A real symmetric matrix is a Hermitian matrix.\medskip

\noindent $\Box$ A Hermitian matrix $M\in \mathcal{M}_{r}(\mathbb{C})$ is positive definite if for all ${\bf z}\in \mathbb{C}^r\backslash\{\bf 0\}$ we have ${\bf z}^HM{\bf z}>0$.
A Hermitian matrix $M\in \mathcal{M}_{r}(\mathbb{C})$ is negative definite if $-M$ is positive definite. 

A Hermitian matrix $M\in \mathcal{M}_{r}(\mathbb{C})$ is positive semidefinite if for all ${\bf z}\in \mathbb{C}^r$ we have ${\bf z}^HM{\bf z}\geq 0$.
A Hermitian matrix $M\in \mathcal{M}_{r}(\mathbb{C})$ is negative semidefinite if $-M$ is positive semidefinite.
\medskip

\noindent $\Box$ The matrix $A\in \mathcal{M}_r(\mathbb{C})$ is similar with the matrix $B\in \mathcal{M}_r(\mathbb{C})$ if exists the invertible matrix $P\in \mathcal{M}_r(\mathbb{C})$ such that $B=P^{-1}AP$. We say that $A$ is similar with $B$ via the matrix $P$. 

The matrix $A$ is {diagonalizable} if it is similar to a diagonal matrix. A matrix that is not diagonalizable is said to be {defective}.\medskip

\noindent $\Box$ For a matrix $A\in \mathcal{M}_r(\mathbb{C})$, we denote by 
$\text{Spec}(A)$ the set of the eigenvalues of $A$. 

{ The algebraic multiplicity} $m(A,\lambda)$ of the eigenvalue $\lambda$ is its multiplicity as a root of the characteristic polynomial.
We denote by $r(A,\lambda)$ {the index} of $\lambda\in \text{Spec}(A)$ and it  is the smallest positive integer $k$ such that $\text{rank} (A-\lambda I_r)^k=\text{rank} (A-\lambda I_r)^{k+1}$.  It is the size of the  largest Jordan block of $A$ with eigenvalue $\lambda$. 
For $\lambda\in \text{Spec}(A)$ we denote by, see \cite{horn}, 
$$w_1(A,\lambda)=r-\text{rank}(A-\lambda I_r),\,\,\,w_k(A,\lambda)=\text{rank}(A-\lambda I_r)^{k-1}-\text{rank}(A-\lambda I_r)^k.$$
The number $w_1(A,\lambda)$ is the number of Jordan blocks of $A$ of all sizes that have eigenvalue $\lambda$ (which is {the geometric multiplicity} of $\lambda$). The number $w_k(A,\lambda)-w_{k+1}(A,\lambda)$ is the number of blocks with eigenvalues $\lambda$ that have size exactly $k$. 
{The Weyr characteristic} of $A$ associated with the eigenvalue $\lambda$ is ${\bf w}(A,\lambda)=\left(w_1(A,\lambda),\dots,w_{r(A,\lambda)}(A,\lambda)\right)^T\in \mathbb{N}^{r(A,\lambda)}.$

The matrices $A,B\in \mathcal{M}_r(\mathbb{C})$ are similar if and only if they have the same distinct eigenvalues, and the same Weyr characteristic associated with each eigenvalue (see Lemma 3.1.18 in \cite{horn}).

The eigenvalue $\lambda$ is {semisimple} if its algebraic and geometric multiplicities are equal; $\lambda$ is a semisimple eigenvalue if and only if every Jordan block corresponding to $\lambda$ is 1-by-1. 

The matrix $A$ is {diagonalizable} if and only if all the eigenvalues are semisimple (all the Jordan blocks are 1-by-1).
\medskip

\noindent $\Box$ Let $r_1,\dots,r_k\in \mathbb{N}^*$, $r=\sum_{i=1}^kr_i$, $M_i\in \mathcal{M}_{r_i}(\mathbb{C})$. The { direct sum} of $M_1,\dots,M_k$ is 
$$M:=M_1\oplus M_2\oplus\dots\oplus M_k=\left(
\begin{array}{cccc}
M_1 & O_{r_1r_2} & \dots & O_{r_1r_k} \\
 O_{r_2r_1} & M_2 & \dots & O_{r_2r_k} \\
\vdots & \vdots & \ddots & \vdots  \\
 O_{r_kr_1} & O_{r_kr_2} & \dots & M_k
\end{array}
\right)\in \mathcal{M}_{r}(\mathbb{C}).$$\medskip

\noindent $\Box$ If $M\in \mathcal{M}_{m\times n}(\mathbb{C})$, $M=[m_{ij}]$, and $N\in \mathcal{M}_{r\times s}(\mathbb{C})$  then { the Kronecker product} is 
$$M\otimes N=
\left(
\begin{array}{ccc}
m_{11}N & ... & m_{1n}N \\
\vdots & \ddots & \vdots \\
m_{m1}N & ... & m_{mn}N
\end{array}
\right)\in \mathcal{M}_{mr\times ns}(\mathbb{C}).$$
If $\{{\bf e}_{1},\dots,{\bf e}_{m}\}
$ is the canonical basis of $\mathbb{C}^m$ and $\{{\bf f}_{1},\dots,{\bf f}_s\}$ is the canonical basis of $\mathbb{C}^s$, then $\left\{{\bf e}_{\alpha}\otimes {\bf f}^T_{\beta}\right\}$ is the canonical basis of $\mathcal{M}_{m\times s}(\mathbb{C})$.\medskip

\noindent $\Box$ A {Jordan block matrix} is a matrix $J_r(\lambda)\in \mathcal{M}_r(\mathbb{C})$, with $r\in \mathbb{N}^*$ and  $\lambda\in \mathbb{C}$, of the form
$$J_1(\lambda)=(\lambda),\,\,\,J_r(\lambda)=\begin{pmatrix}
\lambda & 1 & 0 & \dots & 0 \\
0 & \lambda & 1 & \dots & 0\\
\vdots & \vdots & \vdots & \ddots & \vdots \\
0 & 0 & 0 & \dots & 1 \\
0 & 0 & 0 & \dots & \lambda
\end{pmatrix},\,\,\,\text{if}\,\,\,r>1.$$
The matrix $J_r(0)$ is called { nilpotent Jordan block matrix}. 
We have the following properties:

$\bullet$ $\text{Spec}(J_r(\lambda))=\{\lambda\}$. 

$\bullet$ The matrix $J_r(\lambda)$ is invertible if and only if $\lambda\neq 0$.

$ \bullet$ We have the equality:
\begin{equation}\label{Jordan-block-zero-relation}
J_r(\lambda)=\lambda I_r +J_r(0).
\end{equation}

$\bullet $ $m(J_r(\lambda),\lambda)= r(J_r(\lambda),\lambda)=r$ and ${\bf w}(J_r(\lambda))=(1,1,\dots,1)^T\in \mathbb{N}^r$.

$\bullet$ $J_r(0)$ is a nilpotent matrix with the index of nilpotence equal with $r$; we can write $J_r^r(0)=O_r$.

$\bullet$ The components of $J_r(0)^j$ are 
$[J_r(0)^j]_{\alpha\beta}=\begin{cases}
1 &  \text{ if}\,\, \beta=\alpha+j \\
0 &  \text{ otherwise}
\end{cases}.$

$\bullet$ The matrix $J_s(0)$ is known as the {upper shift matrix} of order $s$. If $M\in \mathcal{M}_{r\times s}(\mathbb{C})$, then in the matrix $MJ_s(0)$ the first column has zeros and the $k$-th column coincide with the $(k-1)$-th column of $M$. If $j\in \{1,\dots, s-1\}$, then the matrix $M(J_s(0))^{j}$ has the columns $1,\dots, j$ with zeros, the $j+1$-th column coincide with the first column of $M$ and so on, the last column coincide with the $(s-j)$-th column of $M$.

$\bullet$ The matrix $J^T_r(0)$ is known that the { lower shift matrix} of order $r$. If $M\in \mathcal{M}_{r\times s}(\mathbb{C})$, then in the matrix $J^T_r(0)M$ appear the elements of $M$ shifted downward by one position, with zeros appearing in the top row.
 If $i\in \{1,\dots, r-1\}$, then $(J^T_r(0))^{i}M$
has the rows $1,\dots, i$ with zeros, the $(i+1)$-th row coincide with the first row of $M$ and so on, the last row coincide with the $(r-i)$-th row of $M$.

$\bullet$ For $\lambda\neq 0$ we have
$J_r(\lambda)^{-1}=\sum\limits_{j=0}^{r-1}\dfrac{(-1)^j}{\lambda^{j+1}}J_r(0)^j=
\begin{pmatrix}
\frac{1}{\lambda} & -\frac{1}{\lambda^2} & \frac{1}{\lambda^3} & \dots &\frac{ (-1)^{r+1}}{\lambda^r} \\
0 & \frac{1}{\lambda} & -\frac{1}{\lambda^2} & \dots & \frac{(-1)^{r}}{\lambda^{r-1}} \\
 \vdots & \vdots & \vdots & \ddots & \vdots \\
 0 & 0 & 0 & \dots & \frac{1}{\lambda}
\end{pmatrix}$.
\medskip

\noindent $\Box$ A matrix $A\in \mathcal{M}_r(\mathbb{C})$ is a {Jordan matrix} if it has the form
\begin{equation}
A=J_{r_1}(\lambda_1)\oplus  J_{r_2}(\lambda_2)\oplus \dots \oplus  J_{r_k}(\lambda_k)
\end{equation}
with $k\in \mathbb{N}^*$, $\sum_{i=1}^kr_i=n$, $\lambda_1,\dots,\lambda_k\in \mathbb{C}$,  and for all $i\in \{1,\dots,k\}$ the matrix $J_{r_i}(\lambda_i)\in \mathcal{M}_{r_i}(\mathbb{C})$ is a Jordan block matrix.

If $\lambda_1=\lambda_2=\dots=\lambda_k=0$, then $A$ is a { nilpotent Jordan matrix}.

The Jordan matrix $A$ has the following properties:

$\bullet$ $\text{Spec}(A)=\{\lambda_1,\dots,\lambda_k\}$.

$\bullet$ $\det A=\prod\limits_{i=1}^k\lambda_i^{r_i}$. $A$ is invertible if and only if $\lambda_i\neq 0$ for all $i\in \{1,\dots,k\}$.

$\bullet$ If $A$ is invertible, then $A^{-1}=J_{r_1}(\lambda_1)^{-1}\oplus J_{r_2}(\lambda_2)^{-1} \oplus \dots \oplus J_{r_k}(\lambda_k)^{-1}$.

$\bullet$ If $k=r$, then $A$ is a diagonal matrix.

\begin{thm}\label{Jordan-canonical-form-theorem}[Jordan canonical form theorem (see \cite{horn})]
Let $A\in \mathcal{M}_r(\mathbb{C})$. There is a nonsingular matrix $P\in \mathcal{M}_r(\mathbb{C})$, natural numbers $r_1,\dots, r_k$ with $\sum_{i=1}^k r_i=r$, and scalars $\lambda_1,\dots,\lambda_k\in \mathbb{C}$ such that
$A=P\cdot J_{r_1}(\lambda_1) \oplus J_{r_2}(\lambda_2)  \oplus \dots \oplus J_{r_k}(\lambda_k)\cdot P^{-1}$.
The Jordan matrix 
$J_A=J_{r_1}(\lambda_1) \oplus J_{r_2}(\lambda_2)  \oplus \dots \oplus J_{r_k}(\lambda_k)$
is uniquely determined by $A$ up to a permutation of Jordan block matrices $J_{r_1}(\lambda_1)$, ..., $J_{r_k}(\lambda_k)$. 

If $A$ is a real matrix and has only real eigenvalues, then $P$ can be chosen to be real. 
\end{thm}

\noindent $\Box$ If $r,s\in \mathbb{N}^*$ and $t\in \{1,\dots,\min\{r,s\}\}$ we introduce the matrices $\mathcal{Y}_t^{[r,s]}\in \mathcal{M}_{r,s}(\mathbb{C})$. 
If $r\leq s$, then 
\begin{equation}\label{Yt-formula}
\mathcal{Y}_t^{[r,s]}=\sum_{\beta=t+s-r}^s (-1)^{s-\beta}{\bf e}_{t+s-\beta}\otimes {\bf f}^T_{\beta},\,\,\,(\mathcal{Y}_t^{[r,s]})_{ij}=
\begin{cases}
(-1)^{s-j} & \text{if}\,\,i+j=t+s \\
0 & \text{otherwise}
\end{cases},
\end{equation}
where $\{{\bf e}_{\alpha}\}$ and $\{{\bf f}_{\beta}\}$ are the canonical bases of $\mathbb{C}^{r}$ and $\mathbb{C}^{s}$.
If $r>s$, then $\mathcal{Y}_t^{[r,s]}=\left(\mathcal{Y}_t^{[s,r]}\right)^T$.

We denote $\mathcal{Y}^{[r,s]}:=\mathcal{Y}_1^{[r,s]}$ and we notice that 
\begin{equation}\label{Y-identity-98}
\mathcal{Y}^{[r,s]}_{t}=
\begin{cases}
(J_r^T(0))^{t-1}\mathcal{Y}^{[r,s]} & \text{if}\,\,\,r\leq s \\
\mathcal{Y}^{[r,s]}(J_s(0))^{t-1} & \text{if}\,\,\,r> s
\end{cases}.
\end{equation}

The matrices $\mathcal{Y}_1^{[r,s]},\dots, \mathcal{Y}_{\max\{r,s\}}^{[r,s]}$ are linearly independent.
The square matrix $\mathcal{Y}^{[r,r]}$ is an anti-diagonal matrix.

For $r=2$ and $s=3$ we have 
$\mathcal{Y}_1^{[2,3]}=\begin{pmatrix}
0 & 0 & 1 \\
0 & -1 & 0
\end{pmatrix},\,\,\mathcal{Y}_2^{[2,3]}=\begin{pmatrix}
0 & 0 & 0 \\
0 & 0 & 1
\end{pmatrix}.$

For $r=3$ and $s=2$ we have  $\mathcal{Y}_1^{[3,2]}=\begin{pmatrix}
0 & 0 \\
0 & -1 \\
1 & 0
\end{pmatrix},\,\,\mathcal{Y}_2^{[3,2]}=\begin{pmatrix}
0 & 0 \\
0 & 0 \\
0 & 1
\end{pmatrix}.
$

\medskip

\noindent $\Box$ The {Pascal matrix} $\Psi^{[r,s]}\in \mathcal{M}_{r,s}(\mathbb{C})$ has the components and the form, see  \cite{brawer-pirovino} and \cite{edelman-strang} (for $r=s$), 
\begin{equation}
\left(\Psi^{[r,s]}\right)_{ij}={{i+j-2}\choose{j-1}},\,\,\,i\in \{1,\dots,r\},\,j\in \{1,\dots, s\},\,\,\,
\Psi^{[r,s]}=\begin{pmatrix}
1 & 1 & 1 & 1 & \dots \\
1 & 2 & 3 & 4 & \dots \\
1 & 3 & 6 & 10 & \dots \\
\vdots & \vdots & \vdots & \vdots & \ddots
\end{pmatrix}.
\end{equation}
The components verify the Pascal's recursion:
\begin{equation}
\begin{cases}
\left(\Psi^{[r,s]}\right)_{1j}=\left(\Psi^{[r,s]}\right)_{i,1}=1,\,\,\,i\in \{1,\dots,r\},\,j\in \{1,\dots, s\},\\
\left(\Psi^{[r,s]}\right)_{ij}=\left(\Psi^{[r,s]}\right)_{i-1,j}+\left(\Psi^{[r,s]}\right)_{i,j-1},\,\,\,i\in \{2,\dots,r\},\,j\in \{2,\dots, s\}.
\end{cases}
\end{equation}

In the paper \cite{zhang-liu} it is considered the {extended generalized Pascal matrix} $\Psi [x,y]$ with the components
\begin{equation}
\left(\Psi^{[r,s]}[x,y]\right)_{ij}= x^{i-j}y^{i+j-2}{{i+j-2}\choose{j-1}}.
\end{equation}
The components verify the recursion
\begin{equation}
\left(\Psi^{[r,s]}[x,y]\right)_{ij}=xy\left(\Psi^{[r,s]}[x,y]\right)_{i-1,j}+\frac{y}{x}\left(\Psi^{[r,s]}[x,y]\right)_{i,j-1}.
\end{equation}
It is noticeable that we have
$\Psi^{[r,s]}=\Psi^{[r,s]}[1,1].$

We are interested in the case $x=1$. We make the notation $\Psi^{[r,s]}[y]=\Psi^{[r,s]}[1,y]$ and we have
\begin{equation}\label{Pascal-generalizat-1}
\left(\Psi^{[r,s]}[y]\right)_{ij}= y^{i+j-2}{{i+j-2}\choose{j-1}},
\,\,\,\Psi^{[r,s]}[y]=\begin{pmatrix}
1 & y & y^2 & y^3 & \dots \\
y & 2y^2 & 3y^3 & 4y^4 & \dots \\
y^2 & 3y^3 & 6y^4 & 10y^5 & \dots \\
\vdots & \vdots & \vdots & \vdots & \ddots
\end{pmatrix}.
\end{equation}

The matrix has the following properties:

\noindent (a) $\left(\Psi^{[r,s]}[y]\right)_{11}=1$.

\noindent (b) For $i\in \{1,\dots,r\}$, $j\in \{1,\dots, s\}$, and $(i,j)\neq (1,1)$ it  is verified the recursion
\begin{equation}
\left(\Psi^{[r,s]}[y]\right)_{ij}=y\left(\Psi^{[r,s]}[y]\right)_{i-1,j}+y\left(\Psi^{[r,s]}[y]\right)_{i,j-1},
\end{equation}
where $\left(\Psi^{[r,s]}[y]\right)_{0,j}=\left(\Psi^{[r,s]}[y]\right)_{i,0}=0$.

\noindent (c) If $r\in \mathbb{N}^*$, then $\det \Psi^{[r,r]}[y]=y^{r(r-1)}$ (see Theorem 8 from \cite{zhang-liu}).

\noindent (d) If $r\in \mathbb{N}^*$ and $y\in \R^*$, then $\Psi^{[r,r]}[y]$ is a positive definite matrix. We can use the Sylvester's criterion and the above result.
\medskip

\noindent $\Box$ We consider $r,s\in \mathbb{N}^*$ and we construct a number of $r\cdot s$ matrices from $\mathcal{M}_{r\times s}(\mathbb{C})$ in a point $y\in \mathbb{C}^*$.

\noindent (i) $\mathcal{X}^{[r,s]}_{11}[y]=y\Psi^{[r,s]}[y]$.

\noindent (ii) $\mathcal{X}^{[r,s]}_{1j}[y]=y\left(O_{r,j-1}\,\,\Psi^{[r,s-j+1]}[y]\right)$, with $j\in \{2,\dots,s\}$.

\noindent (iii) $\mathcal{X}^{[r,s]}_{i1}[y]=y\begin{pmatrix}
O_{i-1,s} \\
\Psi^{[r-i+1,s]}[y]
\end{pmatrix}
$, with $i\in \{2,\dots,r\}$.

\noindent (iv) $\mathcal{X}^{[r,s]}_{ij}[y]=y\begin{pmatrix}
O_{i-1,j-1} & O_{i-1,s-j+1} \\
O_{r-i+1,j-1} & \Psi^{[r-i+1,s-j+1]}[y]
\end{pmatrix}
$, where $i\in \{2,\dots,r\}$ and $j\in \{2,\dots,s\}$.

We denote by $\mathcal{X}^{[r,s]}:=\mathcal{X}^{[r,s]}_{11}$ and we notice that 
\begin{equation}\label{X-identity-998}
\mathcal{X}^{[r,s]}_{ij}=(J_r^T(0))^{i-1}\mathcal{X}^{[r,s]}(J_s(0))^{j-1}.
\end{equation}

We consider the case when $r=2$ and $s=3$. The constructed matrices are:
 $$ \mathcal{X}^{[2,3]}_{11}[y]=
 \left( \begin {array}{ccc} y & y^2 & y^3
\\ \noalign{\medskip} y^2 &2 y^3 &3 y^4
\end {array} \right)
,\,\mathcal{X}^{[2,3]}_{12}[y]=\left( \begin {array}{ccc} 0& y & y^2
\\ \noalign{\medskip}0& y^2 &2 y^3
\end {array} \right),\,
\mathcal{X}^{[2,3]}_{13}[y]=\left( \begin {array}{ccc} 0&0& y\\ \noalign{\medskip}0&0
& y^2 \end {array} \right),
$$
$$\mathcal{X}^{[2,3]}_{21}[y]=\left( \begin {array}{ccc} 0&0&0\\ \noalign{\medskip} y & y^2
& y^3 \end {array} \right),\,
\mathcal{X}^{[2,3]}_{22}[y]=\left(\begin {array}{ccc} 0&0&0\\ \noalign{\medskip}0& y & y^2 \end {array} \right),\,
\mathcal{X}^{[2,3]}_{23}[y]=\left(\begin {array}{ccc} 0&0&0\\ \noalign{\medskip}0&0& y\end {array}
 \right).
$$
We present some properties of the above matrices.

\noindent (a) We have the equalities $ \mathcal{X}^{[r,s]}_{ij}[y]= \left(\mathcal{X}^{[s,r]}_{ji}[y]\right)^T$. 

\noindent (b) If $y\in \R$ and $r\in \mathbb{N}^*$, then for $y>0$ the matrix $ \mathcal{X}^{[r,r]}[y]$ is positive definite and for $y<0$ the matrix $ \mathcal{X}^{[r,r]}[y]$ is negative definite. We use the fact that
$\Psi^{[r,r]}[y]$ is a positive definite matrix.

\end{appendices}

\vspace{0.7cm}

\noindent  {\bf Acknowledgement:} This work was supported by a grant of Ministery of Research and Innovation, CNCS - UEFISCDI, project number PN-III-P4-ID-PCE-2016-0165, within PNCDI III.

\end{document}